\documentclass[sn-mathphys-num]{sn-jnl}
\usepackage{cancel}

\usepackage{setspace}
\usepackage{parskip} 

\usepackage{graphicx}%
\usepackage{multirow}%
\usepackage{amsmath,amssymb,amsfonts}%
\usepackage{amsthm}%
\usepackage{mathrsfs}%
\usepackage[title]{appendix}%
\usepackage{xcolor}%
\usepackage{textcomp}%
\usepackage{manyfoot}%
\usepackage{booktabs}%
\usepackage{algorithm}%
\usepackage{algorithmicx}%
\usepackage{algpseudocode}%
\usepackage{listings}%
\usepackage{lmodern}
\usepackage{placeins}

\usepackage{graphicx}
\usepackage{amsmath,amsfonts,amssymb,mathtools,amsthm}
\usepackage{url}
\usepackage{comment}
\usepackage{color}
\usepackage{algorithm}
\usepackage[noEnd=true]{algpseudocodex}
\usepackage{float}
\usepackage[shortlabels]{enumitem}
\usepackage{multirow}
\usepackage{diagbox}
\usepackage{hyperref}
\usepackage[noabbrev,capitalise]{cleveref}

\usepackage{epstopdf}
\usepackage[normalem]{ulem}

\crefformat{algorithm}{Algorithm~#2#1#3}
\crefrangeformat{algorithm}{Algorithms~#3#1#4 to~#5#2#6}
\crefmultiformat{algorithm}{Algorithms~#2#1#3}%
       { or~#2#1#3}{, #2#1#3}{ or~#2#1#3}

\usepackage{tabularray}
\UseTblrLibrary{varwidth}
\DefTblrTemplate{caption-tag}{default}{\textbf{Table\hspace{0.25em}\thetable}}
\DefTblrTemplate{caption-sep}{default}{:\enskip}
\usepackage[font=small]{caption}

\def\tto{\;{\lower 1pt \hbox{$\rightarrow$}}\kern -10pt
	\hbox{\raise 2pt \hbox{$\rightarrow$}}\;}

\def\Hat{\widehat}

\def\R{\mathbb{R}}

\def\gph{\mbox{\rm gph}\,}

\def\dom{\mbox{\rm dom}}

\newtheorem{theorem}{Theorem}[section]
\newtheorem{proposition}[theorem]{Proposition}
\newtheorem{lemma}[theorem]{Lemma}

\theoremstyle{definition}
\newtheorem{definition}[theorem]{Definition}
\newtheorem{example}[theorem]{Example}

\newtheorem{remark}[theorem]{Remark}

\begin{document}

\title[Article Title]{The Existence and Stability of Generalized Multi-Source Weber Problems}


\author[1,2]{\sur{V. S. T.} \fnm{Long} }\email{vstlong@hcmus.edu.vn}

\author[3]{\sur{N. M.} \fnm{Nam} }\email{mnn3@pdx.edu}

\author[1,2]{\sur{T. H.} \fnm{Tai} }\email{tranhuutai13011@gmail.com}

\author[4]{\sur{C.} \fnm{Tammer} }\email{christiane.tammer@mathematik.uni-halle.de} 


\affil[1]{\orgdiv{Faculty of Mathematics and Computer Science}, \orgname{University of Science}, \orgaddress{\city{Ho Chi Minh City}, \country{Vietnam}}}

\affil[2]{\orgname{Vietnam National University}, \orgaddress{\city{Ho Chi Minh City}, \country{Vietnam}}}

\affil[3]{\orgdiv{Fariborz Maseeh Department of Mathematics and Statistics}, \orgname{Portland State University}, \orgaddress{\city{Portland}, \postcode{97207}, \state{Oregon}, \country{USA}}}

\affil[4]{\orgdiv{Martin-Luther Universität Halle-Wittenberg,} \orgname{Institute of Mathematics,
06099 Halle,} \country{Germany}}


\abstract{	This paper studies the generalized multi-source Weber problem with set-valued targets in the framework of minimal time functions. We first establish the existence of global and local optimal solutions and investigate several qualitative properties of the corresponding solution sets, including closedness, compactness, and conditions ensuring boundedness or unboundedness. Next, we derive explicit Lipschitz continuity properties for the objective function and the associated optimal value function with respect to perturbations of the target sets. We then introduce the global and local solution mappings and study their stability properties from the viewpoint of set-valued analysis. These results provide a quantitative and qualitative sensitivity analysis for the generalized multi-source Weber problem in the setting of minimal time functions.}




\maketitle

\section{Introduction}In the seventeenth century, Fermat posed to Torricelli the problem of finding a point that minimizes the sum of the distances to three given points. This classical problem, now referred to as the Fermat--Torricelli problem, is widely recognized as one of the cornerstones of location science. Beyond its geometric appeal, it has served as a fundamental source of inspiration for a large body of research on facility location and optimization. In particular, many important models have emerged from this line of study, including the multi-source Weber problem (also known as the clustering problem) and the $k$-center problem, which have found numerous applications in operations research, clustering, and data analysis; see, for example, \cite{Aggarwal2014,Bagirov2006,Bagirov2011,Gunther2016,Gunther2018,Jain2010,Kantardzic2011,Kuhn1973,Martini2002,mordukhovich2011applications,Nam2017,Nam2018,Ordin2015}.

In recent years, the multi-source Weber problem (MWP) has been extensively studied from both theoretical and algorithmic perspectives; see, for example, \cite{cuong2020qualitative,cuong2023global,CTWYOptim,Cuong2024}. It is worth emphasizing that MWP is, in general, both nonconvex and nonsmooth. These intrinsic difficulties necessitate the development of specialized optimization techniques, such as DC (difference-of-convex) programming and smoothing methods; see, for example, \cite{AragonArtacho2022,TA1,Tao1986}. At the same time, the $k$-center problem has been generalized from the Euclidean norm setting to the framework of Minkowski gauge functions; see, for example, \cite{LNSYkcenter,LNTV}.

On the other hand, the class of minimal time functions has been extensively studied by various authors (see, e.g., \cite{
	colombo2004subgradient,colombo2010well,durea2016minimal,he2006subdifferentials,jiang2012subdifferential,mordukhovich2010limiting,mordukhovich2011subgradients,wolenski1998proximal} and the references therein) due to its applications in optimization and control. A significant feature of these functions is that they can be viewed as generalized distance functions. Beginning with \cite{nam2013variational}, minimal time functions have been increasingly used in the study of location problems and, more broadly, in variational analysis and optimization. Their applications include directional metric regularity, stability of directional regularity, directional error bounds, and directional Levitin--Polyak well-posedness; see, for instance, \cite{burlicua2023directional,cibulka2020stability,durea2017new}. In addition, this framework has been extended to set-valued constrained optimization through the notion of directional Pareto solutions, which generalizes the classical concept of Pareto efficiency; see \cite{chelmucs2019directional,chelmucs2021stability,cibulka2020stability,chelmus2022exact}. Further results on optimality conditions along this line can be found in \cite{ait2022strict}.

Inspired by the aforementioned developments, we employ the minimal time function as a replacement for the norm in the multi-source Weber problem. We then investigate the existence and stability of global and local optimal solutions.

The paper is organized as follows. In the next section, we present some basic concepts in convex analysis and formulate our problem. In Section \ref{au}, we provide several auxiliary results required for the main analysis. Section \ref{exist} establishes the existence of global and local solutions and explores various properties of the solution sets. In Section \ref{lip}, we study the Lipschitz continuity of the objective and optimal value functions. In Section \ref{solu}, we investigate the stability of solution mappings via upper and inner semicontinuity. The final section provides some concluding remarks.  

Throughout this paper, the Euclidean space $\mathbb{R}^d$ ($d \in \mathbb{N}$) is equipped with the standard inner product $\langle x, y \rangle = \sum_{i=1}^{d} x_i y_i$ and its associated norm $\|x\| = \sqrt{\langle x, x \rangle}$. For any $a \in \mathbb{R}^d$ and $r \geq 0$, the closed and open Euclidean balls centered at $a$ with radius $r$ are denoted by $\mathbb{B}[a; r]$ and $\mathbb{B}(a; r)$, respectively.

For vectors $x=(x^1,\dots,x^k)$ and $y=(y^1,\dots,y^k)$ in the product space $(\mathbb{R}^d)^k=\mathbb{R}^{dk}$, where vectors $x^j,y^j\in\mathbb{R}^d$ for all $j=1,\dots,k$, we use the Euclidean norm
\[
\|x-y\|=\left(\sum_{j=1}^k \|x^j-y^j\|^2\right)^{1/2},
\]
which coincides with the Frobenius norm when $x$ and $y$ are identified with $d\times k$ matrices.

\section{Preliminaries} \label{pre}

This section presents several basic notions and formulates the problem under consideration. For a more comprehensive treatment of the relevant concepts from convex analysis, we refer the reader to \cite{Bauschke,mordukhovich2023easy}.

Unless otherwise specified, the set $F \subset \mathbb{R}^d$ is assumed to be a nonempty compact convex set containing the origin in its interior, i.e., $0 \in \operatorname{int}(F)$, where $0$ denotes the origin of $\mathbb{R}^d$. 
Recall that the \emph{polar set} of $F$ is defined by
\[
F^\circ = \{y \in \mathbb{R}^d \mid \langle y,x\rangle \le 1 \text{ for all } x \in F\}.\]
Set
\[
\|F\| = \sup\{\|x\| \mid x \in F\}\quad\text{and}\quad 
\|F^\circ\| = \sup\{\|y\| \mid y \in F^\circ\}.
\]

A broad class of minimal time functions has been studied in  \cite{	colombo2004subgradient,colombo2010well,durea2016minimal,he2006subdifferentials,jiang2012subdifferential,mordukhovich2010limiting,mordukhovich2011subgradients,wolenski1998proximal} and the references therein. For our purposes, it suffices to consider the following special case.

\begin{definition}
	\label{T_definition}
	Let $\Omega \subset \mathbb{R}^d$ be a nonempty
	compact set (not necessarily convex). The \emph{minimal time function} associated with the sets $F$ and $\Omega$ is defined by
	\begin{equation}
		\label{minimaltime}
		T_F(x,\Omega) = \inf \left\{ t \geq 0 \mid (x + tF) \cap \Omega \neq \emptyset \right\}, \quad x \in \mathbb{R}^d.
	\end{equation}
	The set $F$ is called the \textit{dynamics set} and the set $\Omega$ is called the \textit{target set}.
\end{definition}

\begin{remark}
	\label{T_remark}
	
	\begin{enumerate}
		\item [(a)] Since $0\in {\rm int}(F)$, it follows from Proposition 2.2(i) in \cite{durea2016minimal} that $$T_F(x,\Omega)<+\infty\quad\text{for all }x\in \mathbb{R}^d.$$
		\item [(b)] If $\Omega=\{\omega\}$, then we write $T_F(x,\omega)$ for simplicity.
		\item [(c)] The minimal time function $T_F(\cdot,\Omega)$ admits several useful interpretations. 
		In particular, if $F=\mathbb{B}[0;1]$ (the closed unit ball in $\mathbb{R}^d$), then 
		$T_F(\cdot,\Omega)$ becomes the usual distance function to the set $\Omega$, i.e.,
		\[
		T_F(x,\Omega)=\inf\{\|x-y\| \mid y\in\Omega\}= \operatorname{dist}(x,\Omega).
		\]
		Hence, the minimal time function can be viewed as a natural generalization of the distance function with respect to a general dynamics set $F$.
	\end{enumerate}
\end{remark}

\begin{definition}
	\label{rho_definition}
	The function \( \rho_F : \mathbb{R}^d \to [0, +\infty) \), called the \emph{Minkowski gauge function} associated with the set \( F \subset \mathbb{R}^d \), is defined by
	\begin{equation}\label{gauge}
		\rho_F(x) = \inf \{ t \geq 0 \mid x \in tF \}, \quad x \in \mathbb{R}^d.
	\end{equation}
\end{definition}

Using the minimal time function defined in \eqref{minimaltime}, we now formulate the generalized location problem.
Let $I = \{1, \dots, m\}$ and $J_k= \{1, \dots, k\}$ with $m,k \in \mathbb{N}$. Let $\{\Omega^i\}_{i\in I}$ be a family of nonempty compact (not necessarily convex) subsets of $ \mathbb{R}^d$.  The goal is to find $k$ centers $x^1,\dots,x^k \in \mathbb{R}^d$ that minimize the total minimal time required to serve all target sets $\Omega^i$, where the service time is measured by the minimal time function. 
The corresponding optimization problem is formulated as
\begin{equation} 
	\label{Generalized_Multi_Source_Weber_Problem}
	\min \left\{ f_k(x) = \sum_{i\in I} \min_{j\in J_k} T_F(x^j,\Omega^i) \ \bigg| \ x=(x^1,\dots,x^k)\in\mathbb{R}^{dk} \right\}.
\end{equation}

To avoid trivial cases, we assume that

\begin{equation}
	\label{m_k_condition}
	m\ge 2 \quad \text{and} \quad 1\le k\le m .
\end{equation}

We denote the set of all global optimal solutions of problem \eqref{Generalized_Multi_Source_Weber_Problem} by

\begin{equation}
	\label{Generalized_ Multi_Source_Weber_Problem_Solutions}
	S_k= \left\{ y \in \mathbb{R}^{dk}\mid f_k(y)= \min_{x\in \mathbb{R}^{dk}} f_k(x) \right\} 
\end{equation}

\begin{remark}
	If $\Omega^i=\{a^i\}$ for all $i\in I$, then by \eqref{T_rho_relationship_1} we have
	\[
	T_F(x^j,\Omega^i)=\rho_F(a^i-x^j).
	\]
	In this case, problem \eqref{Generalized_Multi_Source_Weber_Problem} reduces to the formulation involving the Minkowski gauge function, which was studied in \cite{LNTV}.
	Moreover, if $F=\mathbb{B}[0,1]$, then
	\begin{equation*}
		\label{Preliminaries_3}
		T_F(x^j,\Omega^i)=\|x^j-a^i\|,
	\end{equation*}
	and hence problem \eqref{Generalized_Multi_Source_Weber_Problem} becomes the classical multi-source Weber problem, which has been extensively studied from both theoretical and numerical perspectives; see, e.g., \cite{cuong2020qualitative, cuong2023global, CTWYOptim, Cuong2024}. It is worth emphasizing that the problems studied in these papers belong to the class of DC programming problems. However, problem~\eqref{Generalized_Multi_Source_Weber_Problem} does not necessarily belong to this class and is therefore more general.
\end{remark}

\section{Auxiliary results}\label{au}

The following lemmas collects several basic properties of the minimal time function that will be used in the subsequent analysis; for further details, see \cite{long2021new,long2022invariant,long2023directional,mordukhovich2023easy}.

\begin{lemma}
	\label{T_rho_relationship}
	The relationship between $T_F(\cdot,\Omega)$ and the Minkowski gauge $\rho_F$ is given by
	\begin{equation}
		\label{T_rho_relationship_1}
		T_F(x,\Omega) = \inf \left\{ \rho_F(\omega - x) \mid \omega \in \Omega \right\}\;\text{ for } x \in \mathbb{R}^d,
	\end{equation}
	and consequently,
	\begin{equation}
		\label{T_rho_relationship_2}
		T_F(x,\Omega)=\inf\{T_F(x,\omega) \big| \, \omega \in \Omega\}.
	\end{equation}
\end{lemma}
\begin{proof}
	The representation \eqref{T_rho_relationship_1} follows from 
	\cite[Theorem~6.19]{mordukhovich2023easy}. 
	Combining \eqref{minimaltime}, \eqref{gauge}, and \eqref{T_rho_relationship_1}, 
	we obtain \eqref{T_rho_relationship_2}.
\end{proof}

\begin{lemma}
	\label{T_lemma} Consider the minimal time function defined in \eqref{minimaltime}. Then
	for all $x,y \in \mathbb{R}^d$, we have
	\begin{equation} \label{a}
		\frac{T_F(x,y)}{\|F^\circ\|} \leq \| x - y \| \leq \|F\| T_F(x,y),
	\end{equation}
	and consequently,
	\begin{equation}\label{b}
		T_F(x,y) \leq \|F\|\|F^\circ\|T_F(y,x).    
	\end{equation}
\end{lemma}
\begin{proof}
	By Proposition 2.1(c) in \cite{colombo2004subgradient}, we have 
	\[
	\frac{\rho_F(u)}{\|F^\circ\|} \leq \|u\| \leq \|F\| \rho_F(u) \text{ for all } u \in \mathbb{R}^d.
	\]
	Applying this estimate to $u$ and $-u$ yields
	\[
	\rho_F(u) \le \|F^\circ\|\,\|u\|
	\le \|F\|\,\|F^\circ\|\,\rho_F(-u).
	\]
	Replacing $u$ with $y-x$ and using \eqref{T_rho_relationship_1}, we obtain \eqref{a} and \eqref{b}.
\end{proof}

\begin{lemma}\label{Preliminaries_2}
	Consider the minimal time function defined in \eqref{minimaltime}. Then the following properties hold:
	\begin{enumerate}
		\item [{\rm (a)}] $T_F(x,\Omega) = 0$ if and only if $x \in \Omega$.
		\item [{\rm (b)}] If $\Omega$ is convex, then $T_F(\cdot,\Omega)$ is convex. 
		\item [{\rm (c)}] $T_F(\cdot,\Omega)$ is $\|F^\circ\|$-Lipschitz continuous on $\mathbb{R}^d$.
		\item [{\rm (d)}] $T_F(x,y) \leq T_F(x,z)+T_F(z,y)$ for any $x,y,z \in \mathbb{R}^d$.
	\end{enumerate}
\end{lemma}

We conclude this section with a result concerning the target set $\Omega$ and its enlargement.

\begin{lemma} 
	\label{G.O.Solution_1}
	Let $\Omega \subset \mathbb{R}^d$ be a nonempty bounded set and let $r>0$. 
	Then the enlargement of $\Omega$ defined by
	\begin{equation} 
		\label{enlargement}
		\Omega[r] = \left\{ x \in \mathbb{R}^d \mid T_F(x,\Omega) \le r \right\}
	\end{equation}
	is a compact subset of $\mathbb{R}^d$.
\end{lemma}

\begin{proof}
	Since $\Omega$ is bounded, there exists $R>0$ such that 
	$\Omega \subset \mathbb{B}[0;R]$. 
	Observe that $\Omega[r]$ is closed because it is a sublevel set of the 
	continuous function $T_F(\cdot,\Omega)$; see Lemma~\ref{Preliminaries_2}(c). 
	To prove that $\Omega[r]$ is compact, it suffices to show that $\Omega[r]$ is bounded.
	Suppose to the contrary that $\Omega[r]$ is unbounded. 
	Then there exists a sequence $\{x_n\}\subset \Omega[r]$ such that
	\begin{equation}
		\label{key1}
		\lim_{n\to\infty}\|x_n\|=\infty .
	\end{equation}
	For every $n\in\mathbb{N}$, there exists $\omega_n\in\Omega$ such that
	\[
	\begin{aligned}
		T_F(x_n,\Omega)
		&\ge T_F(x_n,\omega_n)-\frac1n 
		\quad (\text{by } \eqref{T_rho_relationship_2})\\
		&\ge \frac{1}{\|F\|}\|\omega_n-x_n\|-\frac1n 
		\quad (\text{by }\eqref{a})\\
		&\ge \frac{1}{\|F\|}\big(\|x_n\|-\|\omega_n\|\big)-\frac1n\\
		&\ge \frac{1}{\|F\|}\big(\|x_n\|-R\big)-\frac1n
		\quad (\text{since } \omega_n\in\Omega\subset\mathbb{B}[0;R]).
	\end{aligned}
	\]
	
	Taking the limit as $n\to\infty$ and using \eqref{key1}, we obtain
	\[
	\lim_{n\to\infty}T_F(x_n,\Omega)
	\ge
	\lim_{n\to\infty}
	\left(
	\frac{1}{\|F\|}(\|x_n\|-R)-\frac1n
	\right)
	=\infty,
	\]
	which contradicts the fact that $T_F(x_n,\Omega)\le r$ for all $n$.
	Hence $\Omega[r]$ must be bounded. This completes the proof. 
\end{proof}

\section{Existence and properties of optimal solutions}\label{exist}
\subsection{Global Optimal Solutions}
In this section, we study the existence and properties of global optimal solutions to problem~\eqref{Generalized_Multi_Source_Weber_Problem}.
We first establish the existence of global optimal solutions for problem \eqref{Generalized_Multi_Source_Weber_Problem}.

\begin{theorem}
	\label{G.O.Solution_2}
	The set of global optimal solutions $S_k$ of problem \eqref{Generalized_Multi_Source_Weber_Problem} is nonempty and closed.
\end{theorem}
\begin{proof}
	Fix any \( i_0 \in I \) and choose \(\omega^{i_0} \in \Omega^{i_0} \). Define
	\begin{equation}\label{key2}
		t^{i_0} = \max_{i \in I} T_F(\omega^{i_0},\Omega^i),    
	\end{equation}
	and
	\begin{equation}\label{key3}
		r^{i_0} = t^{i_0}  + \max_{\omega \in \bigcup_{i \in I} \Omega^i} T_F(\omega,\omega^{i_0}) .
	\end{equation}
	By Lemma \ref{G.O.Solution_1}, the enlargement \( \Omega^{i_0}[r^{i_0}] \) is a compact set. Therefore, it follows from the Weierstrass extreme value theorem (see, e.g., \cite[Theorem 7.9(i)]{mordukhovich2023easy}) that the following optimization problem
	\[
	\min \left\{f_k(x^1,\dots,x^k)\;\Big{|}\; x^j \in \Omega^{i_0}[r^{i_0}] \text{ for all } j \in J_k \right\}
	\]
	admits at least one optimal solution \( \overline{x} = (\overline{x}^1,\dots,\overline{x}^k) \in \mathbb{R}^{dk} \) satisfying
	\[
	T_F(\overline{x}^j,\Omega^{i_0}) \leq r^{i_0} \quad \text{for all } j \in J_k.
	\]
	
	Next, we show that \( \overline{x} \) is a global optimal solution of problem \eqref{Generalized_Multi_Source_Weber_Problem}. Let \( x = (x^1,\dots, x^k) \in \mathbb{R}^{dk} \) be arbitrary. We consider two possible cases:
	\begin{enumerate}
		\item [(a)] \( T_F(x^j,\Omega^{i_0}) \leq r^{i_0} \) for all \( j\in J_k\). 
		\item [(b)]\( T_F(x^{j},\Omega^{i_0}) > r^{i_0} \) for some \( j \in J_k \).
	\end{enumerate}    
	In case (a), it is clear that $	f_k(\overline{x}) \leq f_k(x).$ Now suppose that case (b) holds.  Define the index set
	\begin{equation}
		\label{G.O.Solutions_1_1}
		L = \left\{j \in J_k \mid T_F(x^j,\Omega^{i_0}) > r^{i_0}\right\}
	\end{equation} 
	and construct a new vector $ y = (y^1,\dots,y^k) \in \mathbb{R}^{dk} $ by setting
	\begin{equation}\label{key5}
		y^j = \begin{cases}
			x^j  & \text{if } j \in J_k \setminus L,\\
			\omega^{i_0}  & \text{if } j \in L.
		\end{cases}  
	\end{equation}
	From \eqref{enlargement}, \eqref{key2} and \eqref{key3}, it follows that 
	$ y \in \Omega^{i_0}[r^{i_0}]$, and hence \begin{equation}\label{key66}
		f_k(\overline{x}) \leq f_k(y).
	\end{equation}
	For any fixed \( i \in I \), it is easy to see that 
	\begin{equation}\label{G.O.Solutions_1_2}
		T_F(y^j,\Omega^i) = T_F(x^j,\Omega^i) \quad \text{for every } j \in J_k \setminus L.
	\end{equation}
	Meanwhile, for every \( j \in L \) we have
	\[ \begin{aligned}
		T_F(y^j,\Omega^i) = T_F(\omega^{i_0},\Omega^i) 
		&\leq t^{i_0}\quad(\text{by \eqref{key2} and \eqref{key5})}\\[0.1in]
		&= r^{i_0} - \max_{\omega \in \bigcup_{i \in I} \Omega^i} T_F(\omega,\omega^{i_0}) \quad(\text{by \eqref{key3})} \\[0.1in]
		&\leq T_F(x^j,\Omega^{i_0}) - \max_{\omega \in \bigcup_{i \in I} \Omega^i} T_F(\omega,\omega^{i_0})\quad (\text{by \eqref{G.O.Solutions_1_1}})\\[0.1in]
		&\leq T_F(x^j,\omega^{i_0}) -  \max_{\omega \in \bigcup_{i \in I} \Omega^i} T_F(\omega,\omega^{i_0})\quad(\text{since }\omega^{i_0}\in \Omega^{i_0} ). 
	\end{aligned}
	\]
	Combining this with Lemma~\ref{Preliminaries_2}(d), we obtain that
	for every \(\omega^i\in\Omega^i\),
	\[ \begin{aligned}
		T_F(y^j,\Omega^i) 	&\leq T_F(x^j,\omega^i) + T_F(\omega^i,\omega^{i_0}) - \max_{\omega \in \bigcup_{i \in I} \Omega^i} T_F(\omega,\omega^{i_0}) \\[0.1in]
		&\leq T_F(x^j,\omega^i).  
	\end{aligned}
	\]
	Then~\eqref{T_rho_relationship_1} in Lemma \ref{T_rho_relationship} gives us 
	\[
	T_F(y^j,\Omega^i) \leq T_F(x^j,\Omega^i) \text{ for all } j \in L.
	\]
	This  together with \eqref{key66} and \eqref{G.O.Solutions_1_2} implies that
	\[
	f_k(\overline x)\leq f_k(y) = \sum_{i \in I} \min_{j \in J_k} T_F(y^j,\Omega^i) \leq \sum_{i \in I} \min_{j \in J_k} T_F(x^j,\Omega^i) = f_k(x).
	\]
	Since $x$ is arbitrary, we conclude that \( \overline{x} \in S_k\).  
	
	To verify the closedness of $S_k$, we put 
	$$\alpha=\min_{x\in \mathbb{R}^{dk}}f_k(x)$$
	and consider a sequence
	$\{y_n\} \subset S_k$ such that $y_n\to y \in \mathbb{R}^{dk}$. Since $f_k$ is continuous, we have
	\[
	f_k(y)=\lim_{n\to \infty}f_k\left(y_n\right)=\alpha,
	\]
	which yields  $y \in S_k$. Therefore, $S_k$ is closed. This completes the proof.
\end{proof}

One may naturally ask whether the set of global optimal solutions in 
Theorem~\ref{G.O.Solution_2} is compact. 
When $k=1$, the answer is affirmative, as shown in the following result.

\begin{proposition}
	\label{temp_label_1}
	If $k=1$, then the set of global optimal solutions $S_1$ is compact.
\end{proposition}

\begin{proof}
	By Theorem~\ref{G.O.Solution_2}, the set $S_1$ is nonempty and closed. 
	Thus, it suffices to show that $S_1$ is bounded.
	Suppose on the contrary that $S_1$ is unbounded. 
	Then there exists $y^1\in S_1$ such that
	\begin{equation}
		\label{temp_label_1_1}
		\|y^1\| \ge \|F\|(\alpha+M+2),
	\end{equation}
	where
	\begin{equation}
		\label{temp_label_1_2}
		\alpha = f_1(y^1)=\min_{x\in\mathbb{R}^d} f_1(x)
	\end{equation}
	and
	\begin{equation}
		\label{temp_label_1_3}
		M=\frac{\sup\{\|\omega\|\mid \omega\in\cup_{i\in I}\Omega^i\}}{\|F\|}.
	\end{equation}
	
	Fix any $i\in I$. By~\eqref{T_rho_relationship_1}, there exists 
	$\omega^i\in\Omega^i$ such that
	\begin{equation}
		\begin{aligned}
			\label{mauthuan1}
			T_F(y^1,\Omega^i)
			&\ge T_F(y^1,\omega^i)-1 \\
			&\ge \frac{\|\omega^i-y^1\|}{\|F\|}-1 
			\quad (\text{by Lemma~\ref{T_lemma}(a)})\\
			&\ge \frac{\|y^1\|-\|\omega^i\|}{\|F\|}-1 \\
			&\ge \frac{\|F\|(\alpha+M+2)-\|\omega^i\|}{\|F\|}-1 
			\quad (\text{by \eqref{temp_label_1_1}})\\
			&\ge \alpha+M+2-M-1
			\quad (\text{by \eqref{temp_label_1_3}})\\
			&>\alpha .
		\end{aligned}
	\end{equation}
	However, by \eqref{temp_label_1_2} and the definition of $f_1$, we have
	\[
	\alpha = f_1(y^1) = \sum_{p \in I} T_F(y^1, \Omega^p) \geq T_F(y^1, \Omega^i),
	\]
	which contradicts \eqref{mauthuan1}.
	This means that $S_1$ must be bounded. Hence, $S_1$ is compact.
\end{proof}
\begin{remark}
	If each set $\Omega^i \subset \mathbb{R}^d$ is convex for all $i \in I$, then $T_F(\cdot,\Omega^i)$ is a convex function. According to Proposition 1.44 in \cite{mordukhovich2023easy}, the objective function
	\[ 
	f_1(\cdot)=\sum_{i \in I} T_F(\cdot,\Omega^i)
	\]
	is also a convex function. In this case, $S_1$ is a compact convex set.
\end{remark}

With $k\geq 2$, the set $S_k$ may fail to be compact. This phenomenon is illustrated by the following example.

\begin{example}
	\label{G.O.Solutions_4}
	Consider the case $m=k=2$ and $d=2$. Let
	\[
	\Omega^1=\mathbb{B}[(4,0);2], \quad
	\Omega^2=\mathbb{B}[(6,0);2], 
	\quad \text{and} \quad
	F=\mathbb{B}\!\left[\left(\tfrac12,0\right);1\right].
	\]
	Since
	\[
	f_2(x^1,x^2)\ge 0
	\quad\text{for all } (x^1,x^2)\in(\mathbb{R}^2)^2,
	\]
	and
	\[
	f_2\big((5,0),(x^2_1,x^2_2)\big)=0
	\quad\text{for all } (x^2_1,x^2_2)\in\mathbb{R}^2,
	\]
	the optimal value of problem
	\eqref{Generalized_Multi_Source_Weber_Problem} is equal to $0$.
	Hence
	\[
	\{((5,0),(x^2_1,x^2_2))\mid (x^2_1,x^2_2)\in\mathbb{R}^2\}
	\subset S_2.
	\]
	Consequently, $S_2$ is unbounded.
\end{example}

The next theorem provides a complete characterization of the compactness of $S_k$ when $k \ge 2$.

\begin{theorem}\label{G.O.Solutions_5} Consider problem \eqref{Generalized_Multi_Source_Weber_Problem} with $k\geq2$.
	Then, the following two statements are equivalent:
	\begin{enumerate}
		\item[{\rm (a)}] $S_k$ is compact.
		\item[{\rm (b)}] $\min\limits_{(x^1,\dots,x^k)\in \mathbb{R}^{dk}} f_k\left(x^1,\dots,x^k \right)
		<
		\min\limits_{(x^1,\dots,x^{k-1})\in \mathbb{R}^{d(k-1)}} f_{k-1}\left(x^1,\dots,x^{k-1} \right).$
	\end{enumerate}
	
\end{theorem}

\begin{proof} Note first that $S_k$ is nonempty and closed by Theorem \ref{G.O.Solution_2}. Furthermore, for every $(x^1,\dots,x^k)\in \mathbb{R}^{dk}$ we have
	\begin{equation}\label{G.O.Solutions_5_1}
		\begin{aligned}
			f_k\left(x^1,\dots,x^k \right)
			&= \sum_{i \in I} \min_{j \in J_k} T_F(x^j,\Omega^i)\\
			&\leq \sum_{i \in I} \min_{j \in J_{k-1}} T_F(x^j,\Omega^i)\\&
			= f_{k-1}\left(x^1,\dots,x^{k-1}\right).
		\end{aligned}
	\end{equation}
	Consequently,
	\begin{equation}\label{G.O.Solutions_5_2}
		\min_{(x^1,\dots,x^k)\in \mathbb{R}^{dk}} f_k\left(x^1,\dots,x^k \right)
		\leq
		\min_{(x^1,\dots,x^{k-1})\in \mathbb{R}^{d(k-1)}} f_{k-1}\left(x^1,\dots,x^{k-1} \right).
	\end{equation}
	Therefore, it suffices to prove that the following two statements are equivalent:
	\begin{enumerate}
		\item [{\rm (a')}] $S_k$ is unbounded.
		\item [{\rm (b')}] $\min\limits_{(x^1,\dots,x^k)\in \mathbb{R}^{dk}} f_k\left(x^1,\dots,x^k \right)
		=
		\min\limits_{(x^1,\dots,x^{k-1})\in \mathbb{R}^{d(k-1)}} f_{k-1}\left(x^1,\dots,x^{k-1} \right).$    
	\end{enumerate}
	
	(a') $\Rightarrow$ (b'). Assume that \(S_k\) is unbounded. Then there exists \((y^1,\dots,y^k)\in S_k\) such that
	\[
	\sum_{j=1}^k \|y^j\|^2 \geq k(\|F\|(\alpha+M+2))^2
	\]
	where 
	\begin{equation}\label{key7}
		\alpha = f_k\left(y^1,\dots,y^k\right) = \min_{(x^1,\dots,x^k)\in \mathbb{R}^{dk}} f_k\left(x^1,\dots,x^k\right)    
	\end{equation}
	and
	\begin{equation*}\label{key77}
		M = \frac{\sup\{\|\omega\| \mid \omega \in \cup_{i\in I}\Omega^i\}}{\|F\|}.
	\end{equation*}
	Without loss of generality, we assume that
	\begin{equation}\label{key6}
		\|y^k\| \geq \|F\|(\alpha+M+2).    
	\end{equation}
	By \eqref{key7}, the definition of $f_k$ and arguments similar to those used in \eqref{mauthuan1}, we obtain
	\[T_F(y^k,\Omega^i)> \min_{j \in J_k}T_F(y^j,\Omega^i) \text{ for all } i \in I,
	\]
	which yields
	\begin{equation*}\label{G.O.Solutions_5_4}
		\min_{j \in J_k} T_F(y^j,\Omega^i)
		= \min_{j \in J_{k-1}} T_F(y^j,\Omega^i) \text{ for all } i \in I. 
	\end{equation*}
	Therefore, we have
	\begin{equation*} \label{G.O.Solutions_5_5}
		\begin{aligned}
			f_k\left(y^1,\dots,y^k\right)=\sum_{i \in I}\min_{j \in J_k}T_F(y^j,\Omega^i) 
			= \sum_{i \in I}\min_{j \in J_{k-1}}T_F(y^j,\Omega^i)=f_{k-1}\left(y^1,\dots,y^{k-1}\right).
		\end{aligned}
	\end{equation*}
	This together with \eqref{key7} implies that
	$$	\begin{aligned}
			\min_{(x^1,\dots,x^{k-1})\in\mathbb{R}^{d(k-1)}}f_{k-1}\left(x^1,\dots,x^{k-1} \right)&\leq f_{k-1}\left(y^1,\dots,y^{k-1}\right)\\
			&=f_k\left(y^1,\dots,y^k\right)\\
			&=\min_{(x^1,\dots,x^k)\in\mathbb{R}^{dk}}f_k\left((x^1,\dots,x^k) \right).
		\end{aligned}$$
	Combining this with \eqref{G.O.Solutions_5_2}, we obtain the equality in (b').

	(b') $\Rightarrow$ (a'). Suppose that (b') holds. By Theorem \ref{G.O.Solution_2},
	we can find $(y^1,\dots,y^{k-1})\in \mathbb{R}^{d(k-1)}$ such that
	\[
	f_{k-1}\left(y^1,\dots,y^{k-1}\right)
	=\min_{(x^1,\dots,x^{k-1})\in\mathbb{R}^{d(k-1)}}f_{k-1}\left(x^1,\dots,x^{k-1}\right).
	\]
	Then for any $y^* \in \mathbb{R}^d$, it follows from \eqref{G.O.Solutions_5_1} and (b') that
	\[
	\begin{aligned}
		f_k\left(y^1,\dots,y^{k-1},y^*\right)
		&\leq f_{k-1}\left(y^1,\dots,y^{k-1} \right)\\&=\min_{(x^1,\dots,x^{k-1})\in\mathbb{R}^{d(k-1)}}f_{k-1}\left(x^1,\dots,x^{k-1}\right)\\
		&=\min_{(x^1,\dots,x^{k})\in\mathbb{R}^{dk}}f_k\left(x^1,\dots,x^k\right).
	\end{aligned}
	\]
	Consequently, we obtain
	\[
	f_k\left(y^1,\dots,y^{k-1},y^*\right)
	=\min_{(x^1,\dots,x^{k})\in\mathbb{R}^{dk}}f_k\left(x^1,\dots,x^k\right).
	\]
	Since $y^*\in \mathbb{R}^d$ is arbitrary, the set $S_k$ is unbounded. This means that (a') holds and therefore completes the proof.
\end{proof}

Recently, the notions of natural clustering and attraction sets have been widely used in \cite{cuong2020qualitative,cuong2023global,CTWYOptim,Cuong2024} 
to study multi-source Weber and $k$-center problems. Motivated by this, we introduce the following definitions.

\begin{definition}
	\label{G.O.Solutions_7}
	Let $x=(x^1,\dots,x^k)\in \mathbb{R}^{dk}$. Set $A^0=\emptyset$, and for each $j\in J_k$, define
	\[
	A^j=\left\{\Omega^i \in \left(\{\Omega^i\}_{i \in I}\setminus\bigcup_{q=0}^{j-1}A^q\right)\,\Bigg|\,
	T_F(x^j,\Omega^i)=\min_{r\in J_k}T_F(x^r,\Omega^i)\right\}.
	\]
	The family $\{A^1,\dots,A^k\}$ is called the \textit{natural clustering} associated with $x$. 
\end{definition}

\begin{definition}
	\label{G.O.Solutions_8}
	Let $x=(x^1,\dots,x^k)\in \mathbb{R}^{dk}$. We say that the component $x^j$ of $x$, with $j\in J_k$, is \textit{attractive with respect to} $\{\Omega^i\}_{i\in I}$ if the set
	\[
	A_{x^j}
	=\left\{\Omega^i \in \{\Omega^i\}_{i \in I}\,\Big|\,
	T_F(x^j,\Omega^i)=\min_{r \in J_k}T_F(x^r,\Omega^i)\right\}
	\]
	is nonempty.
\end{definition}

It follows from the above definitions that
\begin{equation}
\label{A^j_A_x^j_relationship}
A^j= A_{x^j}\setminus\bigg(\bigcup_{q=0}^{j-1} A^q\bigg).
\end{equation}

The following theorem shows that a center located sufficiently far from the target sets can be removed without affecting optimality.

\begin{theorem}
	\label{G.O.Solutions_9} Let $\overline{x}=(\overline{x}^1,\dots,\overline{x}^k)\in S_k$ with $k\ge 2$. 
	Assume that there exist $i_0 \in I$,  $\omega^{i_0} \in \Omega^{i_0}$ and $j \in J_k$ such that
	\[
	\overline{x}^j \notin \Omega^{i_0}[r^{i_0}],
	\]
	where 
	\begin{equation}\label{ri0}
		r^{i_0}= (\|F\|\,\|F^\circ\|)^2\left(\sum_{p\in I}T_F(\omega^{i_0},\Omega^p)
		+\frac{\max\big\{T_F(\omega^{i_0},\omega)\mid \omega \in \bigcup_{p \in I}\Omega^p\big\}}{\|F\|\,\|F^\circ\|}\right)   
	\end{equation}
	and $\Omega^{i_0}[r^{i_0}]$ is defined in \eqref{enlargement}. 
	Then
	\[
	\overline{x}^{-j}= (\overline{x}^1,\dots,\overline{x}^{j-1},\overline{x}^{j+1},\dots,\overline{x}^{k}) \in S_{k-1},
	\]
	and the set $S_k$ is unbounded.
\end{theorem}

\begin{proof} Observe first that for any permutation $\sigma$ of $J_k$, the vector 
	$\overline{x}^\sigma = (\overline{x}^{\sigma(1)},\dots,\overline{x}^{\sigma(k)})$ 
	is also an element of  $S_k$. Therefore, we may assume without loss of generality that 
	$j=k$ and $\overline{x}^k \not\in \Omega^{i_0}[r^{i_0}]$. Fix any $i \in I$. Then for every $\varepsilon > 0$, there exists $\omega_\varepsilon^i \in \Omega^i$ such that 
	\begin{equation}
		\label{G.O.Solutions_9_1}
		\begin{aligned}
			T_F(\overline{x}^k,\Omega^i)&\geq T_F(\overline{x}^k,\omega^i_{\varepsilon})-\varepsilon\quad(\text{by \eqref{T_rho_relationship_2} of Lemma~\ref{T_rho_relationship}})\\
&\geq\frac{T_F(\omega^i_{\varepsilon},\overline{x}^k)}{\|F\|\|F^\circ\|} -\varepsilon \quad (\text{by Lemma~\ref{T_lemma}(b)})\\
			&\geq\frac{T_F(\omega^{i_0},\overline{x}^k)-T_F(\omega^{i_0},\omega^i_{\varepsilon})}{\|F\|\|F^\circ\|}-\varepsilon \quad (\text{by Lemma~\ref{Preliminaries_2}(d)})\\
			&\geq\frac{T_F(\overline{x}^k,\omega^{i_0})}{(\|F\|\|F^\circ\|)^2}-\frac{T_F(\omega^{i_0},\omega^i_\varepsilon)}{\|F\|\|F^\circ\|}-\varepsilon \quad (\text{by Lemma~\ref{T_lemma}(b)})\\
			&\geq \frac{T_F(\overline{x}^k,\Omega^{i_0})}{(\|F\|\|F^\circ\|)^2} -\frac{T_F(\omega^{i_0},\omega^i_\varepsilon)}{\|F\|\|F^\circ\|}-\varepsilon \quad (\text{by~\eqref{T_rho_relationship_1}})\\
			&>\frac{r^{i_0}}{(\|F\|\|F^\circ\|)^2}-\frac{T_F(\omega^{i_0},\omega^i_\varepsilon)}{\|F\|\|F^\circ\|}-\varepsilon \quad (\text{since } \overline{x}^k \not\in \Omega^{i_0}[r^{i_0}])\\
			&=\sum_{p \in I}T_F(\omega^{i_0},\Omega^p)+\frac{\max\big\{T_F(\omega^{i_0},\omega)\mid \omega \in \bigcup_{p \in I}\Omega^p\big\}}{\|F\|\|F^\circ\|}-\frac{T_F(\omega^{i_0},\omega^i_\varepsilon)}{\|F\|\|F^\circ\|}-\varepsilon \quad (\text{by~\eqref{ri0}})\\
			&\geq \sum_{p \in I}T_F(\omega^{i_0},\Omega^p)-\varepsilon.
		\end{aligned}
	\end{equation}
	Since $i \in I$ is arbitrary, passing to the limit as $\varepsilon \to 0$, we obtain
	\begin{equation}
		\label{G.O.Solutions_9_2}
		T_F(\overline{x}^k,\Omega^i)> \sum_{p \in I}T_F(\omega^{i_0},\Omega^p) \text{ for all } i \in I.
	\end{equation}
	Next, we show that $A_{\overline{x}^k}=\emptyset$ (see Definition \ref{G.O.Solutions_8}). Indeed, suppose by contradiction that there is 
    \begin{equation}\label{def47}
            \Omega^* \in A_{\overline{x}^k}=\left\{\Omega^i \in \{\Omega^i\}_{i \in I}\mid
	T_F(\bar x^k,\Omega^i)=\min_{j \in J_k}T_F(\bar x^j,\Omega^i)\right\}.
    \end{equation}
   Define \(\omega^*=(\omega^{i_0},\dots,\omega^{i_0})\in\mathbb{R}^{dk}\). Then one has
	\[
	\begin{aligned}
		f_k(\overline{x})&\geq \min_{j \in J_k}T_F(\overline{x}^j,\Omega^*)\quad(\text{by \eqref{Generalized_Multi_Source_Weber_Problem}})\\
		&=T_F(\overline{x}^k,\Omega^*)\quad(\text{by \ref{def47}})\\
		&>\sum_{p \in I}T_F(\omega^{i_0},\Omega^p)\quad(\text{by~\eqref{G.O.Solutions_9_2}})\\
		&=f_k(\omega^*).
	\end{aligned}
	\]
	This contradicts the optimality of $\overline{x}$ for problem \eqref{Generalized_Multi_Source_Weber_Problem}.
	Thus, we must have $A_{\overline{x}^k} = \emptyset$. It follows that for every $i\in I$,
	\[
	\min_{j \in J_k}T_F(\overline{x}^j,\Omega^i)=\min_{j\in J_{k-1}}T_F(\overline{x}^j,\Omega^i),\]
	and hence
	\begin{equation}\label{key99}
		f_{k}(\overline{x}^1,\dots,\overline{x}^k)=f_{k-1}(\overline{x}^1,\dots,\overline{x}^{k-1}).
	\end{equation}
	Moreover, since
	
	\begin{equation}\label{key100}
		f_k(\overline{x})=\min_{(x^1,\dots,x^k)\in \mathbb{R}^{dk}} f_k\left(x^1,\dots,x^k \right)
		\leq
		\min_{(x^1,\dots,x^{k-1})\in \mathbb{R}^{d(k-1)}} f_{k-1}\left(x^1,\dots,x^{k-1} \right),
	\end{equation}
	we conclude that $\overline{x}^{-k}=(\overline{x}^1,\cdots,\overline{x}^{k-1})\in S_{k-1}.$ Combining this with \eqref{key99} and \eqref{key100}, we obtain
	\[
	\min_{x \in \mathbb{R}^{dk}} f_k(x)=\min_{x \in \mathbb{R}^{d(k-1)}} f_{k-1}(x).
	\]
	Therefore, condition (b) of Theorem \ref{G.O.Solutions_5} does not hold. According to this theorem, $S_k$ is not compact. Since $S_k$ is closed by Theorem~\ref{G.O.Solution_2}, it follows that $S_k$ is unbounded. This completes the proof.
\end{proof}

The following theorem provides a sufficient condition ensuring the compactness of the global optimal solution set $S_k$.

\begin{theorem}
	\label{G.O.Solutions_11}
	Consider problem \eqref{Generalized_Multi_Source_Weber_Problem} and
	assume that $k\geq 2$ and the sets $\Omega^1,\dots,\Omega^m$ are pairwise disjoint, i.e.,
	$$\Omega^i \cap \Omega^p = \emptyset\quad\text{for all }i,p \in I\text{ with }i \neq p.$$  
	Then, the global optimal solution set $S_k$ is compact.
	Moreover, for any $x=(x^1,\dots,x^k) \in S_k$, any $i\in I$ and any $\omega^i\in \Omega^i$, one has
	\begin{equation}\label{key999}
	x^j \in \Omega^{i}[r^{i}] \quad \text{for all } j \in J_k,    
	\end{equation}
		where $r^{i}$ and $\Omega^{i}[r^{i}]$ are defined as in \eqref{ri0} and \eqref{enlargement}, respectively. 
\end{theorem}

\begin{proof}
We first show that
\begin{equation}
\label{G.O.Solutions_11_1}
\min_{(x^1,\dots,x^k)\in\mathbb{R}^{dk}} f_k(x^1,\dots,x^k)
<
\min_{(x^1,\dots,x^{k-1})\in\mathbb{R}^{d(k-1)}} f_{k-1}(x^1,\dots,x^{k-1}).
\end{equation}

Let $\overline{x}^{-k}=(\overline{x}^1,\dots,\overline{x}^{k-1})\in S_{k-1}$ be arbitrary. We claim that there exists an index $i_*\in I$ such that
\begin{equation}
\label{G.O.Solutions_11_2}
\min_{j\in J_{k-1}}T_F(\overline{x}^j,\Omega^{i_*})>0.
\end{equation}
Indeed, suppose on the contrary that for every $i\in I$, one has
\[
\min_{j\in J_{k-1}}T_F(\overline{x}^j,\Omega^i)=0.
\]
Then, for each $i\in I$, there exists $j(i)\in J_{k-1}$ such that
\[
T_F(\overline{x}^{j(i)},\Omega^i)=0.
\]
By Lemma~\ref{Preliminaries_2}(a), this implies that
\[
\overline{x}^{j(i)}\in\Omega^i \quad \text{for all } i\in I.
\]
Since the sets $\Omega^1,\dots,\Omega^m$ are pairwise disjoint, each point $\overline{x}^j$ can belong to at most one of them. Hence the $k-1$ centers $\overline{x}^1,\dots,\overline{x}^{k-1}$ can belong to at most $k-1$ target sets. This is impossible because $m\ge k$ by \eqref{m_k_condition}. Therefore, \eqref{G.O.Solutions_11_2} must hold. Now, choose any $\omega^{i_*}\in\Omega^{i_*}$ and define
\[
\Hat{x}=(\overline{x}^1,\dots,\overline{x}^{k-1},\omega^{i_*})\in\mathbb{R}^{dk}.
\]
Then it follows from \eqref{G.O.Solutions_11_2} that
\begin{equation}\label{key33}
\min_{j\in J_k}T_F(\Hat{x}^j,\Omega^{i_*})=T_F(\omega^{i_*},\Omega^{i_*})=0
<
\min_{j\in J_{k-1}}T_F(\overline{x}^j,\Omega^{i_*}).    
\end{equation}
 Moreover, for every $i\in I\setminus\{i_*\}$, one has
 \begin{equation}\label{key34}
    \min_{j\in J_k}T_F(\Hat{x}^j,\Omega^i)
\le
\min_{j\in J_{k-1}}T_F(\overline{x}^j,\Omega^i).
\end{equation}
Therefore,
\[
\begin{aligned}
\min_{(x^1,\dots,x^k)\in\mathbb{R}^{dk}} f_k(x^1,\dots,x^{k})
&\le f_k(\overline{x}^1,\dots,\overline{x}^{k-1},\omega^{i_*}) \\
&= \sum_{i\in I}\min_{j\in J_k}T_F(\Hat{x}^j,\Omega^i) \\
&< \sum_{i\in I}\min_{j\in J_{k-1}}T_F(\overline{x}^j,\Omega^i)\quad(\text{by \eqref{key33} and \eqref{key34}}) \\
&= f_{k-1}(\overline{x}^{-k}) \\
&= \min_{(x^1,\dots,x^{k-1})\in\mathbb{R}^{d(k-1)}} f_{k-1}(x^1,\dots,x^{k-1}) \quad(\text{since $\overline{x}^{-k}\in S_{k-1}$}).
\end{aligned}
\]

This proves \eqref{G.O.Solutions_11_1}, which is exactly condition \rm{(b)} in Theorem~\ref{G.O.Solutions_5}. Hence, by Theorem~\ref{G.O.Solutions_5}, we conclude that $S_k$ is compact.

It remains to show that \eqref{key999} holds. Let $ x=( x^1,\dots, x^k)\in S_k$ be arbitrary.  Suppose on the contrary that there exist $i_0\in I$ and $j_*\in J_k$ such that
\[
 x^{j_*}\notin \Omega^{i_0}[r^{i_0}].
\]
Then, we deduce from Theorem~\ref{G.O.Solutions_9} that $S_k$ is unbounded. This contradicts the fact that $S_k$ is compact. Therefore, for any $i\in I$ one has
\[
x^j\in \Omega^{i}[r^{i}] \quad \text{for all } j\in J_k,
\]
which completes the proof.
\end{proof}

\subsection{Local optimal solutions}

In this subsection, we investigate the qualitative properties of local optimal solutions to problem ~\eqref{Generalized_Multi_Source_Weber_Problem}. The results presented here extend the existing findings in \cite{cuong2020qualitative,cuong2023global,CTWYOptim,Cuong2024,LNSYkcenter,LNTV} and related literature by generalizing the distance measure from the standard Euclidean norm and Minkowski gauge function to the minimal time function.

\begin{definition}\label{deflocal}
	A point $\overline{x}\in\mathbb{R}^{dk}$ is called a \textit{local optimal solution} of problem~\eqref{Generalized_Multi_Source_Weber_Problem} if there exists $\varepsilon>0$ such that
	\[
	f_k(\overline{x})\leq f_k(x)\quad \text{for all } x\in\mathbb{B}(\overline{x};\varepsilon).
	\]
	Moreover, if the above inequality is strict for all
	\(x\in\mathbb{B}(\overline{x};\varepsilon)\setminus\{\overline{x}\},\)
	then $\overline{x}$ is said to be a \textit{strict local optimal solution}. We denote by $S_k^{{\rm loc}}$ the set of all local optimal solutions of problem~\eqref{Generalized_Multi_Source_Weber_Problem}.
\end{definition}

For $\overline{x}=(\overline{x}^1,\dots,\overline{x}^k)\in \mathbb{R}^{dk}$, define
\begin{equation}
	\label{L.O.Solution_1}
	I(\overline{x}^j)=\big\{\,i\in I \,\big|\, \Omega^i \in A_{\overline{x}^j} \big\} \text{ and } 
	J_i(\overline{x})=\{j\in J_k \mid \Omega^i \in A_{\overline{x}^j}\}.
\end{equation}

The following lemma plays a crucial role in establishing the main results of this section.

\begin{lemma}
	\label{L.O.Solution_2}
	Let $\overline{x}=(\overline{x}^1,\dots,\overline{x}^k)\in \mathbb{R}^{dk}$ be given. Suppose that for every $i\in I$ the index set $J_i(\overline{x})$ defined in \eqref{L.O.Solution_1} is a singleton. Then there exists $\varepsilon>0$ such that for any 
	$x=(x^1,\dots,x^k)$ satisfying $\|x^j-\overline{x}^j\|<\varepsilon$ for all $j \in J_k$, \color{black}the following properties hold:
	\begin{enumerate}
		\item [\rm {(a)}] $J_i(x)=J_i(\overline{x})$ for all $i \in I$.
		\item [\rm {(b)}] $I(\overline{x}^j)=I(x^j)$ and $A_{x^j}=A_{\overline{x}^j}$ for all $j \in J_k$.
	\end{enumerate}
\end{lemma}

\begin{proof}
	(a) For each $i\in I$, let $j(i) \in J_k$ denote the unique element of $J_i(\overline{x})$. By the definition of $J_i(\overline{x})$, we have
	\begin{equation}\label{key101}
		T_F(\overline{x}^{j(i)},\Omega^i) < T_F(\overline{x}^j,\Omega^i) \quad \text{for all } j \in J_k\setminus\{j(i)\}.    
	\end{equation}
	Set
	\begin{equation*}
		\delta_i(\overline{x}) = \min_{j\in J_k \setminus \{j(i)\}} \bigl( T_F(\overline{x}^j,{\Omega^i}) - T_F(\overline{x}^{j(i)},\Omega^i) \bigr) > 0   
	\end{equation*}
	and
	\begin{equation}\label{key103}
		\varepsilon_i = \frac{\delta_i(\overline{x})}{4\|F^\circ\|}.
	\end{equation}
	Let $x=(x^1,\dots,x^k)\in \mathbb{R}^{dk}$ satisfy 
	\begin{equation}\label{key102}
		\|x^j-\overline{x}^j\|<\varepsilon_i\quad\text{for all }j \in J_k.
	\end{equation}
	By Lemma \ref{Preliminaries_2}(c), the function $T_F(\cdot,\Omega^i)$ is $\|F^\circ\|$-Lipschitz continuous on $\mathbb{R}^d$, and hence
	\begin{equation}
		\label{L.O.Solution_2_1}
		\begin{aligned}
			T_F(x^{j(i)},\Omega^i) &\le T_F(\overline{x}^{j(i)},\Omega^i) + \|F^\circ\|\|x^{j(i)}-\overline{x}^{j(i)}\|\\ &< T_F(\overline{x}^{j(i)},\Omega^i) + \frac{\delta_i(\overline{x})}{2}. \quad(\text{by \eqref{key103} and \eqref{key102}})
		\end{aligned}
	\end{equation}
	Furthermore, for every $j\in J_k\setminus\{j(i)\}$, we also obtain from the Lipschitz continuity, \eqref{key103} and \eqref{key102} that
	\begin{equation*}
		\label{L.O.Solution_2_2}
		\begin{aligned}
			T_F(x^j,\Omega^i) &\ge T_F(\overline{x}^j,{\Omega^i}) - \|F^\circ\|\|x^j-\overline{x}^j\|\\ &> T_F(\overline{x}^j,\Omega^i) - \frac{\delta_i(\overline{x})}{2}.
		\end{aligned}
	\end{equation*}
	Combining this with \eqref{key101} and \eqref{L.O.Solution_2_1}, we get
	\[
	\begin{aligned}
		T_F(x^{j(i)},\Omega^i)-T_F(x^j,\Omega^i)&<T_F(\overline{x}^{j(i)},\Omega^i)+\frac{\delta_i(\overline{x})}{2}-T_F(\overline{x}^j,\Omega^i) + \frac{\delta_i(\overline{x})}{2}\\
		&\leq 0.
	\end{aligned}
	\]
	This means that
	\begin{equation}\label{aboveinequality}
		T_F(x^{j(i)},\Omega^i) < T_F(x^j,\Omega^i) \quad \text{for all } j \in J_k\setminus \{j(i)\}.    
	\end{equation}
	Set $\varepsilon = \min_{i\in I} \varepsilon_i$. Then, for any $x=(x^1,\dots,x^k)$ satisfying $\|x^j-\overline{x}^j\|<\varepsilon$ for all $j \in J_k$, the inequality \eqref{aboveinequality} holds for every $i\in I$. Consequently, $$J_i(x) = \{j(i)\} = J_i(\overline{x}).$$
	
	(b) For every  $j \in J_k$, one has
	\[
	\begin{aligned}
		I(\overline{x}^j)
		&= \{\,i \in I \mid \Omega^i \in A_{\overline{x}^j}\,\}\\
		&= \{\,i \in I \mid j \in J_i(\overline{x})\,\}\\
		&= \{\,i \in I \mid j \in J_i(x)\, \} \quad(\text{by (a)})\\
		&= \{\,i \in I \mid \Omega^i \in A_{x^j}\,\}\\ 
		&= I(x^j).
	\end{aligned}
	\]
	Hence, $A_{\overline{x}^j} = A_{x^j}$.
\end{proof}

Let $\overline{x}=(\overline{x}^1,\dots,\overline{x}^k)\in \mathbb{R}^{dk}$ be such that 
$A_{\overline{x}^j}\neq \emptyset$ for some $j\in J_k$. 
Let $I(\overline{x}^j)$ be defined as in \eqref{L.O.Solution_1}. We associate with $\overline{x}^j$ the following generalized single-source Weber problem:

\begin{equation}\label{single_source_Weber_problem}
	\min \biggl\{\,f^j_{1}(x)=\sum_{i \in I(\overline{x}^j)} T_F(x,\Omega^i)\,\bigg|\, x\in \mathbb{R}^d \biggr\}.
\end{equation}

We conclude this subsection by establishing the necessary and sufficient conditions for the existence of local optimal
solutions to problem \eqref{Generalized_Multi_Source_Weber_Problem}.

\begin{theorem} \label{L.O.Solution_3}
	Let $\overline{x}=(\overline{x}^1,\dots,\overline{x}^k)\in \mathbb{R}^{dk}$, and suppose that $J_i(\overline{x})$ is a singleton for every $i \in I$. Consider the following assertions:
	
	\begin{enumerate}
		\item[\rm {(a)}] For each $j \in J_k$, if $A_{\overline{x}^j}$ is nonempty, then $\overline{x}^j$ is a (global) optimal solution of the generalized single-source Weber problem \eqref{single_source_Weber_problem} defined on the data set $A_{\overline{x}^j}$.
		\item[\rm{(b)}] $\overline{x}=(\overline{x}^1,\dots,\overline{x}^k)$ is a local optimal solution to problem \eqref{Generalized_Multi_Source_Weber_Problem}.
	\end{enumerate}
	
	Then (a) implies (b). The converse implication also holds if, in addition, each $\Omega^i$ is convex for all $i \in I$.
	
\end{theorem}

\begin{proof} We first prove that $(a) \Rightarrow (b)$. Let $\overline{x}=(\overline{x}^1,\dots,\overline{x}^k)\in \mathbb{R}^{dk}$ satisfy the hypotheses. For each $i \in I$, let $j(i)\in J_k$ denote the unique element of $J_i(\overline{x})$. By Lemma \ref{L.O.Solution_2}(a), there exists $\varepsilon>0$ such that for any $x=(x^1,\dots,x^k)\in \mathbb{R}^{dk}$ satisfying $\|x^j-\overline{x}^j\|<\varepsilon$ for all $j \in J_k$, we have $$J_i(x)=J_i(\overline{x})=\{j(i)\}\quad\text{for every }i \in I.$$ It follows that
	
	\begin{equation}
		\label{L.O.Solution_3_1}T_F(x^{j(i)},\Omega^i)=\min_{j \in J_k}T_F(x^j,\Omega^i).
	\end{equation}
	By assumption, for every $j \in J_k$ such that $A_{\overline{x}^j} \neq \emptyset$, we have
	
	\begin{equation}
		\label{L.O.Solution_3_2}
		\sum_{i\in I(\overline{x}^j)}T_F(\overline{x}^j,\Omega^i)\leq\sum_{i\in I(\overline{x}^j)}T_F(x^j,\Omega^i).
	\end{equation}
	We partition the index set $I$ into disjoint subsets $I = \bigcup_{j \in J_k} I(\overline{x}^j)$ and adopt the convention that$$\sum_{i \in I(\overline{x}^j)} T_F(x^j, \Omega^i) = 0 \quad \text{whenever } I(\overline{x}^j) = \emptyset.$$
	Then we have
	\[
	\begin{aligned}
		f_k(x) 
		&= \sum_{i \in I} \min_{j \in J_k} T_F(x^j,\Omega^i) \\
		&= \sum_{i \in I} T_F(x^{j(i)},\Omega^i) \quad (\text{by \eqref{L.O.Solution_3_1}})\\ 
		&= \sum_{j \in J_k} \sum_{i \in I(\overline{x}^j)} T_F(x^{j(i)},\Omega^i) \\
		&= \sum_{j \in J_k} \sum_{i \in I(\overline{x}^j)} T_F(x^j,\Omega^i) \\
		&\geq \sum_{j \in J_k} \sum_{i \in I(\overline{x}^j)} T_F(\overline{x}^{j},\Omega^i) \quad (\text{by \eqref{L.O.Solution_3_2}})\\
		&= \sum_{i \in I} T_F(\overline{x}^{j(i)},\Omega^i) \\
		&= \sum_{i \in I} \min_{j\in J_k}T_F(\overline{x}^j,\Omega^i) \quad (\text{by \eqref{L.O.Solution_3_1}})\\
		&= f_k(\overline{x}).
	\end{aligned}
	\]
	Thus, for every $x=(x^1,\dots,x^k) \in \mathbb{R}^{dk}$ satisfying $\|x^j-\overline{x}^j \|<\varepsilon$ for all $j \in J_k$, we have
	\begin{equation}
		\label{L.O.Solution_3_3}
		f_k(\overline{x}) \leq f_k(x).
	\end{equation}
	If $\|x-\overline{x}\|<\varepsilon$, then
	\[
	\|x^j-\overline{x}^j\|\le \|x-\overline{x}\|<\varepsilon
	\quad \text{for all } j \in J_k.
	\]
	Hence, by \eqref{L.O.Solution_3_3}, we obtain
	\[
	f_k(\overline{x})\le f_k(x)
	\quad \text{for all } x\in\mathbb{B}(\overline{x};\varepsilon).
	\]
	Therefore, $\overline{x}\in S_k^{{\rm loc}}$.

	Next, we prove that the converse implication $(b)\Rightarrow(a)$ also holds under the additional assumption that each $\Omega^i$ is convex for all $i\in I$. To this end, let $\overline{x}=(\overline{x}^1,\dots,\overline{x}^k)\in \mathbb{R}^{dk}$ be a local optimal solution of problem~\eqref{Generalized_Multi_Source_Weber_Problem}. This means that there exists $\varepsilon_1>0$ such that
	\[
	f_k(\overline{x})\leq f_k(x) \text{ for all } x \in \mathbb{B}(\overline{x},\varepsilon_1),
	\]
	which is equivalent to
	\begin{equation}
		\label{L.O.Solution_3_4}
		f_k(\overline{x})\leq f_k(x) \text{ for all } x\in \mathbb{R}^{dk} \text{ such that } \|x-\overline{x}\|<\varepsilon_1.
	\end{equation}
	Since $J_i(\overline{x})$ is singleton for any $i \in I$, from Lemma \ref{L.O.Solution_2}(b) it follows that there exists $\varepsilon_2>0$ such that for every $x = (x^1,\dots,x^k)$ satisfying
    \[
    \|x^j-\overline{x}^j\|<\varepsilon_2\quad \text{for all }j \in J_k,
    \]
    we have
	\begin{equation}
		\label{L.O.Solution_3_5}
		A_{\overline{x}^j}=A_{x^j} \text{ and } I(\overline{x}^j)=I(x^j).
	\end{equation}
	It is easy to verify that
	\[
	\|x-\overline{x}\|\leq \sqrt{k}\,\max_{j\in J_k}\|x^j-\overline{x}^j\|.
	\]
	Moreover, for any $K>0$, one has
	\[
	\|x^j-\overline{x}^j\|<K \quad \text{for all } j \in J_k
	\Longleftrightarrow
	\max_{j\in J_k}\|x^j-\overline{x}^j\|<K.
	\]
	Set
	\[
	\varepsilon=\min\left\{\frac{\varepsilon_1}{\sqrt{k}},\varepsilon_2\right\}.
	\]
	Then, for any $x=(x^1,\dots,x^k)\in \mathbb{R}^{dk}$ satisfying
	\[
	\|x^j-\overline{x}^j\|<\varepsilon \quad \text{for all } j \in J_k,
	\]
	the vector $x$ satisfies both \eqref{L.O.Solution_3_4} and \eqref{L.O.Solution_3_5}. Thus
	\begin{equation}
		\label{L.O.Solution_3_6}
		A_{\overline{x}^j}=A_{x^j},\quad I(\overline{x}^j)=I(x^j),\quad\text{and}\quad     f_k(\overline{x})\leq f_k(x).
	\end{equation}
	Assume that $A_{\overline{x}^{j_0}} \neq \emptyset$ for some $j_0 \in J_k$. Let $y \in \mathbb{R}^d$ be arbitrary such that
	\[
	\|y-\overline{x}^{j_0}\|<\varepsilon.
	\]
	Define $\Hat{x}=(\Hat{x}^1,\dots,\Hat{x}^k)\in \mathbb{R}^{dk}$ by
    \begin{equation}
    \label{Hat_x_def}
    	\Hat{x}^j=
	\begin{cases}
		y & \text{if } j=j_0,\\
		\overline{x}^j & \text{if } j \in J_k\setminus\{j_0\}.
	\end{cases}
    \end{equation}
    Since $\|\Hat{x}^j-\overline{x}^j\|<\varepsilon$ for all $j \in J_k$, we obtain
	\begin{equation}
		\label{L.O.Solution_3_7}
		\begin{aligned}
			f_k(\overline{x})&=\sum_{i \in I(\overline{x}^{j_0})}\left(\min_{j \in J_k}T_F(\overline{x}^j,\Omega^i) \right)+\sum_{i \not\in I(\overline{x}^{j_0})}\left(\min_{j \in J_k}T_F(\overline{x}^j,\Omega^i) \right)\\
			&=\sum_{i \in I(\overline{x}^{j_0})}T_F(\overline{x}^{j_0},\Omega^i) +\sum_{i \not\in I(\overline{x}^{j_0})}\left(\min_{j \in J_k\setminus \{j_0\}}T_F(\overline{x}^j,\Omega^i) \right),
		\end{aligned}
	\end{equation}
	and
	\begin{equation}
		\label{L.O.Solution_3_8}
		\begin{aligned}
			f_k(\Hat{x})    &=\sum_{i \in I(\Hat{x}^{j_0})}\left(\min_{j \in J_k}T_F(\Hat{x}^j,\Omega^i) \right)+\sum_{i \not\in I(\Hat{x}^{j_0})}\left(\min_{j \in J_k}T_F(\Hat{x}^j,\Omega^i) \right)\\
            &=\sum_{i \in I(\Hat{x}^{j_0})}T_F(\Hat{x}^{j_0},\Omega^i) +\sum_{i \not\in I(\Hat{x}^{j_0})}\left(\min_{j \in J_k\setminus\{j_0\}}T_F(\Hat{x}^j,\Omega^i) \right)\\
			&=\sum_{i \in I(\Hat{x}^{j_0})}T_F(y,\Omega^i)+\sum_{i \not\in I(\Hat{x}^{j_0})}\left(\min_{j \in J_k\setminus\{j_0\}}T_F(\overline{x}^j,\Omega^i) \right) \quad (\text{by \eqref{Hat_x_def}})\\
			&=\sum_{i \in I(\overline{x}^{j_0})}T_F(y,{\Omega^i})+\sum_{i \not\in I(\overline{x}^{j_0})}\left(\min_{j \in J_k \setminus \{j_0\}}T_F(\overline{x}^j,{\Omega^i}) \right) \quad (\text{by \eqref{L.O.Solution_3_6}}).
		\end{aligned}
	\end{equation}
    It also follows from \eqref{L.O.Solution_3_6} that
	\[
	f_k(\overline{x}) \le f_k(\Hat{x}).
	\]
	Combining this with \eqref{L.O.Solution_3_7} and \eqref{L.O.Solution_3_8}, we get
	\[
	\sum_{i \in I(\overline{x}^{j_0})} T_F(\overline{x}^{j_0},\Omega^i) \le \sum_{i \in I(\overline{x}^{j_0})} T_F(y,\Omega^i) 
	\quad \text{for all } y \in \mathbb{R}^d \text{ such that } \|y-\overline{x}^{j_0}\|<\varepsilon.
	\]
	Hence, $\overline{x}^{j_0}$ is a local optimal solution of the single-source Weber problem associated with the data set $A_{\overline{x}^{j_0}}$. Under the additional assumption that each $\Omega^i$ is convex, it follows from Lemma~\ref{Preliminaries_2}(b) that the function $T_F(\cdot,\Omega^i)$ is convex for every $i\in I$. Therefore, the objective function of problem~\eqref{single_source_Weber_problem} is convex, and hence every local optimal solution is also a global one. Consequently, $\overline{x}^{j_0}$ is a global optimal solution of problem~\eqref{single_source_Weber_problem}. This completes the proof.
\end{proof}

\section{Lipschitz properties of the objective and optimal value functions}\label{lip}

This section is devoted to investigating the Lipschitz continuity of the objective function and the optimal value function associated with problem \eqref{Generalized_Multi_Source_Weber_Problem}. To facilitate our analysis, we first introduce the necessary notation, including relevant sets and metrics.

Let 
\[
\mathcal{K}=\{A \subset \mathbb{R}^d \mid A \text{ is nonempty and compact}\}
\]
be the family of all nonempty compact subsets of $\mathbb{R}^d$. For any $A,B \in \mathcal{K}$, define
\begin{equation}\label{dk}
	d_{\mathcal{K}}(A,B)=\max\left\{ \sup_{a\in A}\inf_{b\in B}\|a-b\|,\ \sup_{b\in B}\inf_{a\in A}\|a-b\| \right\}.    
\end{equation}

Furthermore, for any elements $\Omega = (\Omega^1, \dots, \Omega^m)$ and $\Hat\Omega = (\Hat\Omega^1, \dots, \Hat\Omega^m)$ in the product space $\mathcal{K}^m$, we define the metric $d_{\mathcal{K}^m}$ by
\begin{equation}\label{dkm}
	d_{\mathcal{K}^m}(\Omega, \Hat\Omega) = \max_{1 \leq i \leq m} d_{\mathcal{K}}(\Omega^i, \Hat\Omega^i).  
\end{equation}

We then consider the product space
\[
\mathcal{X}= \mathbb{R}^{dk}\times \mathcal{K}^m,
\]
equipped with the metric
\begin{equation}\label{dX}
	d_\mathcal{X}((x,\Omega),(\Hat x,\Hat \Omega))=\| x-\Hat x\|+d_{\mathcal{K}^m}(\Omega,\Hat \Omega),
\end{equation}
where $x=(x^1,\dots,x^k)$, $\Hat x=(\Hat x^1,\dots,\Hat x^k)$, $\Omega=(\Omega^1,\dots,\Omega^m)$, and $\Hat \Omega=(\Hat \Omega^1,\dots,\Hat \Omega^m)$.

It is straightforward to verify that $(\mathcal{K}, d_{\mathcal{K}})$, $(\mathcal{K}^m, d_{\mathcal{K}^m})$, and $(\mathcal{X}, d_{\mathcal{X}})$ are well-defined metric spaces.

\begin{remark}From this section onward, we use the notation $f_k(x, \Omega)$ for $(x, \Omega) \in \mathcal{X}$ to represent the objective function of problem \eqref{Generalized_Multi_Source_Weber_Problem}. This formulation explicitly accounts for the dependence of the objective value on both the centers and the target sets, which is more convenient for the subsequent sensitivity analysis.\end{remark}

The following results establishes the Lipschitz continuity of the component functions with respect to both the variables and the target sets.

\begin{proposition}
	\label{O.V.F_Lipschitz_1}
	For any $(i,j)\in I\times J$, define $\varphi_{i,j}:\mathcal{X}\to \mathbb{R}$ by
	\[
	\varphi_{i,j}(x,\Omega)=T_F(x^j,\Omega^i)\quad\text{for }(x,\Omega)=((x^1,\dots,x^m),(\Omega^1,\dots\Omega^m))\in\mathcal{X}.
	\]
	Then $\varphi_{i,j}$ is $\|F^\circ\|$-Lipschitz continuous on $\mathcal{X}$, that is,
	\[
	\big|\varphi_{i,j}(x,\Omega)-\varphi_{i,j}(\Hat x,\Hat \Omega)\big|
	\leq \|F^\circ\|\, d_\mathcal{X}\big((x,\Omega),(\Hat x,\Hat \Omega)\big),
	\]
	for all $(x,\Omega),(\Hat{x},\Hat{\Omega}) \in \mathcal{X}$.
\end{proposition}

\begin{proof}
	Fix $(i,j)\in I\times J$ and let $(x,\Omega),(\Hat x,\Hat \Omega)\in \mathcal{X}$. 
	For any $\varepsilon>0$, there exists $\Hat{\omega}_{\varepsilon}^i\in \Hat{\Omega}_i$ such that
	\[
	T_F(\Hat x^j,\Hat \Omega^i)\ge T_F(\Hat x^j,\Hat \omega^i_{\varepsilon})-\varepsilon.
	\]
	Let $\omega^i\in \Omega^i$ be arbitrary. Then
	
	$$\begin{aligned} \varphi_{i,j}(x,\Omega)-\varphi_{i,j}(\Hat x,\Hat \Omega) &= T_F(x^j,\Omega^i)-T_F(\Hat x^j,\Hat \Omega^i)\\ &\leq T_F(x^j,\omega^i)-(T_F(\Hat x^j,\Hat \omega^i_{\varepsilon})-\varepsilon) \quad(\text{by \eqref{T_rho_relationship_2}})\\ 
		&= \big[T_F(x^j,\omega^i)-T_F(\Hat x^j,\omega^i)\big] + \big[T_F(\Hat x^j,\omega^i)-T_F(\Hat x^j,\Hat \omega^i_{\varepsilon})\big] + \varepsilon\\
		&\leq T_F(x^j,\Hat{x}^j) + T_F(\Hat{\omega}^i_\varepsilon,\omega^i) + \varepsilon \quad(\text{by Lemma \ref{Preliminaries_2}(d)})\\
		&\leq \|F^\circ\|\,\|x^j-\Hat x^j\| + \|F^\circ\|\,\|\omega^i-\Hat \omega^i_{\varepsilon}\| + \varepsilon \quad(\text{by \eqref{a}}). \end{aligned}$$
	
	By taking the infimum over $\omega^i \in \Omega^i$ and the supremum over $\hat{\omega} \in \hat{\Omega}_i$, and then invoking \eqref{dk} and \eqref{dX}, we obtain
	\[
	\begin{aligned}
		\varphi_{i,j}(x,\Omega)-\varphi_{i,j}(\Hat x,\Hat \Omega)
		&\leq \|F^\circ\|\Big(\inf_{\omega^i\in \Omega^i}\|\omega^i-\Hat \omega^i_{\varepsilon}\|
		+\|x^j-\Hat x^j\|\Big) + \varepsilon\\
		&\leq \|F^\circ\|\Big(
		\sup_{\Hat \omega \in \Hat \Omega^i}\inf_{\omega \in \Omega^i}\|\omega-\Hat \omega\|
		+\|x-\Hat x\|
		\Big) + \varepsilon\\
		&\leq \|F^\circ\|\Big(
		d_{\mathcal{K}}(\Omega^i,\Hat \Omega^i)
		+\|x-\Hat x\|
		\Big) + \varepsilon\\
		&\leq \|F^\circ\|\, d_{\mathcal{X}}\big((x,\Omega),(\Hat x,\Hat \Omega)\big) + \varepsilon.
	\end{aligned}
	\]
	Since $\varepsilon>0$ is arbitrary, it follows that
	\[
	\varphi_{i,j}(x,\Omega)-\varphi_{i,j}(\Hat x,\Hat \Omega)
	\leq \|F^\circ\|\, d_{\mathcal{X}}\big((x,\Omega),(\Hat x,\Hat \Omega)\big).
	\]
	Changing the roles of $(\Hat x,\Hat \Omega)$ and $(x,\Omega)$, we also have
	\[
	\varphi_{i,j}(\Hat x,\Hat \Omega)-\varphi_{i,j}(x,\Omega) \leq \|F^\circ\|d_{\mathcal{X}}((\Hat x,\Hat \Omega),(x,\Omega)).
	\]
	Hence,
	\begin{equation*}
		\label{O.V.F_Lipschitz_1_0}
		|\varphi_{i,j}(x,\Omega)-\varphi_{i,j}(\Hat x,\Hat \Omega)| \leq \|F^\circ\|d_{\mathcal{X}}((x,\Omega),(\Hat x,\Hat \Omega)),
	\end{equation*}
	for all $(x,\Omega) \text{ and } (\Hat x,\Hat \Omega)$ in  $\mathcal{X}$.
\end{proof}

\begin{proposition}
	\label{O.V.F_Lipschitz_2}
	For each $i \in I$, the function $\psi_i(\cdot,\Omega):\mathcal{X}\to\mathbb{R}$ defined by
	\[
	\psi_i(x,\Omega)=\min_{j \in J_k}\varphi_{i,j}(x,\Omega)
	\]
	is $\|F^\circ\|$-Lipschitz continuous on $\mathcal{X}$.
\end{proposition}

\begin{proof}
	Let $(x,\Omega),(\Hat{x},\Hat{\Omega}) \in \mathcal{X}$. 
	Choose $p \in J_k$ such that
	\[
	\psi_i(\Hat{x},\Hat{\Omega})
	=\varphi_{i,p}(\Hat{x},\Hat{\Omega}).
	\]
	Then we have
	\[
	\begin{aligned}
		\psi_i(x,\Omega)-\psi_i(\Hat{x},\Hat{\Omega})
		&\leq \varphi_{i,p}(x,\Omega)-\varphi_{i,p}(\Hat{x},\Hat{\Omega})\\
		&\leq \|F^\circ\|\, d_{\mathcal{X}}\big((x,\Omega),(\Hat{x},\Hat{\Omega})\big),
	\end{aligned}
	\]
	where the last inequality follows from Proposition~\ref{O.V.F_Lipschitz_1}.
	Interchanging the roles of $(x,\Omega)$ and $(\Hat{x},\Hat{\Omega})$, we also obtain
	\[
	\psi_i(\Hat{x},\Hat{\Omega})-\psi_i(x,\Omega)
	\leq \|F^\circ\|\, d_{\mathcal{X}}\big((x,\Omega),(\Hat{x},\Hat{\Omega})\big).
	\]
	Therefore,
	\[
	|\psi_i(x,\Omega)-\psi_i(\Hat{x},\Hat{\Omega})|
	\leq \|F^\circ\|\, d_{\mathcal{X}}\big((x,\Omega),(\Hat{x},\Hat{\Omega})\big),
	\]
	for all $(x,\Omega),(\Hat{x},\Hat{\Omega}) \in \mathcal{X}$.
\end{proof}

We now establish the Lipschitz continuity of the objective function $f_k$.

\begin{theorem}
	\label{O.V.F_Lipschitz_3}
	The objective function $f_k$ is $m\|F^\circ\|$-Lipschitz continuous on $\mathcal{X}$.
\end{theorem}

\begin{proof}
	Recall that
	\[
	f_k(x,\Omega)=\sum_{i\in I}\psi_i(x,\Omega).
	\]
	Let $(x,\Omega),(\Hat{x},\Hat{\Omega})\in \mathcal{X}$. By Proposition \ref{O.V.F_Lipschitz_2}, for each $i\in I$, we have
	\[
	\left|\psi_i(x,\Omega)-\psi_i(\Hat{x},\Hat{\Omega})\right|
	\le \|F^\circ\|\, d_{\mathcal X}\big((x,\Omega),(\Hat{x},\Hat{\Omega})\big).
	\]
	Therefore,
	\[
	\begin{aligned}
		\left|f_k(x,\Omega)-f_k(\Hat{x},\Hat{\Omega})\right|
		&= \left|\sum_{i\in I}\psi_i(x,\Omega)-\sum_{i\in I}\psi_i(\Hat{x},\Hat{\Omega})\right| \\
		&\le \sum_{i\in I}\left|\psi_i(x,\Omega)-\psi_i(\Hat{x},\Hat{\Omega})\right| \\
		&\le \sum_{i\in I}\|F^\circ\|\, d_{\mathcal X}\big((x,\Omega),(\Hat{x},\Hat{\Omega})\big) \\
		&= m\|F^\circ\|\, d_{\mathcal X}\big((x,\Omega),(\Hat{x},\Hat{\Omega})\big).
	\end{aligned}
	\]
	Hence, $f_k$ is $m\|F^\circ\|$-Lipschitz continuous on $\mathcal{X}$.
\end{proof}

To complete the stability analysis of the generalized multi-source Weber problem, we now establish the Lipschitz continuity of its optimal value function with respect to the parameter~$\Omega$.

\begin{theorem}
	\label{O.V.F_Lipschitz_4}
	The optimal value function $v_k:\mathcal{K}^m\to \mathbb{R}$ associated with problem~\eqref{Generalized_Multi_Source_Weber_Problem}, defined by
	\begin{equation}
		\label{v_k_definition}
		v_k(\Omega)=\min\{f_k(x,\Omega)\mid x\in \mathbb{R}^{dk}\},
	\end{equation}
	is $m\|F^\circ\|$-Lipschitz continuous on $\mathcal{K}^m$.
\end{theorem}

\begin{proof}
	Let $\Omega,\Hat{\Omega}\in \mathcal{K}^m$ be arbitrary. For any $i\in I$ and $j\in J_k$, it follows from Proposition \ref{O.V.F_Lipschitz_1_0} that
	\[
	\varphi_{i,j}(x,\Omega)
	\leq \varphi_{i,j}(x,\Hat{\Omega})
	+\|F^\circ\|\, d_{\mathcal X}\big((x,\Omega),(x,\Hat{\Omega})\big)
	\]
	for every $x\in \mathbb{R}^{dk}$. By \eqref{dkm} and \eqref{dX}, we have
	\[
	d_{\mathcal K^m}(\Omega,\Hat{\Omega})=d_{\mathcal X}\big((x,\Omega),(x,\Hat{\Omega})\big),
	\]
	and hence
	\[
	\varphi_{i,j}(x,\Omega)
	\leq \varphi_{i,j}(x,\Hat{\Omega})
	+\|F^\circ\|\, d_{\mathcal K^m}(\Omega,\Hat{\Omega})
	\quad \text{for all } x\in \mathbb{R}^{dk}.
	\]
	Taking the minimum over $j\in J_k$ on both sides yields
	\[
	\psi_i(x,\Omega)
	\leq \psi_i(x,\Hat{\Omega})
	+\|F^\circ\|\, d_{\mathcal K^m}(\Omega,\Hat{\Omega})
	\quad \text{for all } x\in \mathbb{R}^{dk}.
	\]
	Summing over $i\in I$, we arrive at
	\begin{equation}
		\label{o.v.f_Lipschitz_1}
		f_k(x,\Omega)-f_k(x,\Hat{\Omega})
		\leq m\|F^\circ\|\, d_{\mathcal K^m}(\Omega,\Hat{\Omega})
		\quad \text{for all } x\in \mathbb{R}^{dk}.
	\end{equation}
	
	Now let $\varepsilon>0$ be arbitrary. By \eqref{v_k_definition}, there exists $y\in \mathbb{R}^{dk}$ such that
	\[
	f_k(y,\Hat{\Omega})\leq v_k(\Hat{\Omega})+\varepsilon.
	\]
	On the other hand, again by \eqref{v_k_definition}, we have
	\[
	v_k(\Omega)\leq f_k(y,\Omega).
	\]
	Therefore,
	\begin{equation}
		\label{o.v.f_Lipschitz_2}
		v_k(\Omega)-v_k(\Hat{\Omega})
		\leq f_k(y,\Omega)-f_k(y,\Hat{\Omega})+\varepsilon.
	\end{equation}
	Combining \eqref{o.v.f_Lipschitz_1} and \eqref{o.v.f_Lipschitz_2}, we obtain
	\[
	v_k(\Omega)-v_k(\Hat{\Omega})
	\leq m\|F^\circ\|\, d_{\mathcal K^m}(\Omega,\Hat{\Omega})+\varepsilon.
	\]
	Since $\varepsilon>0$ is arbitrary, it follows that
	\[
	v_k(\Omega)-v_k(\Hat{\Omega})
	\leq m\|F^\circ\|\, d_{\mathcal K^m}(\Omega,\Hat{\Omega}).
	\]
	By interchanging the roles of $\Omega$ and $\Hat{\Omega}$, we also have
	\[
	v_k(\Hat{\Omega})-v_k(\Omega)
	\leq m\|F^\circ\|\, d_{\mathcal K^m}(\Omega,\Hat{\Omega}).
	\]
	Hence,
	\[
	|v_k(\Omega)-v_k(\Hat{\Omega})|
	\leq m\|F^\circ\|\, d_{\mathcal K^m}(\Omega,\Hat{\Omega}).
	\]
	This means that $v_k$ is $m\|F^\circ\|$-Lipschitz continuous on $\mathcal{K}^m$.
\end{proof}

\section{Stability of solution mappings}
\label{solu}
In this subsection, we investigate several properties of both the global and local solution mappings.

\subsection{Global solution mappings}

Since $(\mathcal{K}^m,d_{\mathcal{K}^m})$; see \eqref{dkm}, is the parameter space of problem~\eqref{Generalized_Multi_Source_Weber_Problem}, we now recall several notions from set-valued analysis that will be used in the subsequent analysis.

Let $G:(\mathcal{K}^m,d_{\mathcal{K}^m})\rightrightarrows \mathbb{R}^{dk}$ be a set-valued mapping. The \textit{domain} and \textit{graph} of $G$ are defined, respectively, by
\[
\dom(G)=\left\{\Omega\in\mathcal{K}^m \mid G(\Omega)\neq\emptyset\right\},
\]
and
\[
\gph(G)=\left\{(x,\Omega)\in\mathcal{X}\mid x\in G(\Omega)\right\}.
\]

\begin{definition}
	Let $G:(\mathcal{K}^m,d_{\mathcal{K}^m})\rightrightarrows \mathbb{R}^{dk}$ be a set-valued mapping and let $\Hat \Omega\in \dom(G)$. We say that
	\begin{enumerate}
		\item  $G$ is \textit{closed} at $\Hat \Omega$ if for every sequence
		\[
		(\Omega_n,x_n)\in \gph(G)
		\]
		such that
		\[
		\Omega_n\to \Hat \Omega\quad \text{in }(\mathcal{K}^m,d_{\mathcal{K}^m})
		\quad \text{and} \quad
		x_n\to \overline{x}\quad \text{in }\mathbb{R}^{dk},
		\]
		one has
		\[
		\overline{x}\in G(\Hat \Omega).
		\]
		
		\item $G$ is \textit{compact-valued} on $\mathcal{K}^m$ if $G(\Omega)$ is a compact subset of $\mathbb{R}^{dk}$ for every $\Omega\in \mathcal{K}^m$.
		\item We say that $G$ is \textit{upper semicontinuous} at $\Hat \Omega$ if for every open set $V\subset \mathbb{R}^{dk}$ satisfying
		\[
		G(\Hat \Omega)\subset V,
		\]
		there exists a neighborhood $U$ of $\Hat \Omega$ such that
		\[
		G(\Omega)\subset V \quad \text{for all } \Omega\in U.
		\]
		
		\item $G$ is \textit{inner semicontinuous} at $(\Hat \Omega,\overline{x})\in \gph(G)$ if for every open set $V\subset \mathbb{R}^{dk}$ with $\overline{x}\in V$, there exists a neighborhood $U$ of $\Hat \Omega$ such that
		\[
		G(\Omega)\cap V\neq \emptyset \quad \text{for all } \Omega\in U.
		\]
	\end{enumerate}
\end{definition}

\begin{definition}
	\label{G.S.Mapping_1}
	The set-valued mapping $S_k:(\mathcal{K}^m,d_{\mathcal{K}^m})
	\rightrightarrows \mathbb{R}^{dk}$, defined by
	\begin{equation}
		\label{g.l.b_S_k_definition}
		S_k(\Omega)=\left\{x\in\mathbb{R}^{dk}\mid f_k(x,\Omega)=v_k(\Omega)\right\},
		\quad \Omega=(\Omega^1,\dots,\Omega^m)\in\mathcal{K}^m,
	\end{equation}
	is called the \textit{global solution mapping} of problem
	\eqref{Generalized_Multi_Source_Weber_Problem}.
\end{definition}

\begin{proposition}
	\label{G.S.Mapping_3}
	$S_k$ is closed at any $\Omega\in \mathcal{K}^m$. 
\end{proposition}

\begin{proof}
	Let $(x,\Omega)\in\mathcal X$ and let $\{(x_n,\Omega_n)\}\subset \gph(S_k)$ be a sequence such that
	\[
	(x_n,\Omega_n)\longrightarrow{}(x,\Omega).
	\]
	Since $(x_n,\Omega_n)\in \gph(S_k)$, one has
	\[
	f_k(x_n,\Omega_n)=v_k(\Omega_n)\quad \text{for all } n\in\mathbb N.
	\]
	It follows from Theorem \ref{O.V.F_Lipschitz_3} and Theorem \ref{O.V.F_Lipschitz_4} that
	$f_k$ is continuous on $(\mathcal X,d_{\mathcal X})$ and $v_k$ is continuous on
	$(\mathcal K^m,d_{\mathcal K^m})$. Therefore,
	\[
	f_k(x,\Omega)
	=\lim_{n\to\infty} f_k(x_n,\Omega_n)
	=\lim_{n\to\infty} v_k(\Omega_n)
	=v_k(\Omega).
	\]
	Hence, $x\in S_k(\Omega)$. The proof is complete.
\end{proof}

\begin{theorem}
	\label{G.S.Mapping_4}
	Let $\Hat\Omega\in \mathcal{K}^m$ and let $U_0$ be a neighborhood of $\Hat\Omega$. Suppose that there exists a compact set $W\subset \mathbb{R}^{dk}$ such that
	\[
	S_k(\Omega)\subset W \quad \text{for all } \Omega\in U_0.
	\]
	Then the global solution mapping $S_k$ is upper semicontinuous at $\Hat\Omega$.
\end{theorem}

\begin{proof}
	Suppose on the contrary that $S_k$ is not upper semicontinuous at $\Hat\Omega$. Then, by definition, there exists an open set $V\subset \mathbb{R}^{dk}$ such that
	\[
	S_k(\Hat\Omega)\subset V,
	\]
	but for every neighborhood $U$ of $\Hat\Omega$, one can find $\Omega\in U$ such that
	\[
	S_k(\Omega)\not\subset V.
	\]
	
	For each $n\in\mathbb{N}$, define
	\[
	U_n=\mathbb{B}_{d_{\mathcal{K}^m}}\left(\Hat\Omega;\frac{1}{n}\right)\cap U_0,
	\]
	where
	\[
	\mathbb{B}_{d_{\mathcal{K}^m}}\left(\Hat\Omega;\frac{1}{n}\right)
	=\left\{\Theta\in\mathcal{K}^m \mid d_{\mathcal{K}^m}(\Theta,\Hat\Omega)<\frac{1}{n}\right\}.
	\]
	Then $U_n$ is a neighborhood of $\Hat\Omega$. Hence, for each $n\in\mathbb{N}$, there exist $\Omega_n\in U_n$ and $x_n\in S_k(\Omega_n)$ such that
	\[
	x_n\notin V.
	\]
	By construction of $U_n$, it follows that $\Omega_n\to \Hat\Omega$ in $(\mathcal{K}^m,d_{\mathcal{K}^m})$. Since $\Omega_n\in U_0$, we have
	\[
	x_n\in S_k(\Omega_n)\subset W \quad \text{for all } n\in\mathbb{N}.
	\]
	Because $W$ is compact, there exists a subsequence $(x_{n'})$ converging to some $\overline{x}\in W$.
	
	Moreover, since $x_{n'}\in S_k(\Omega_{n'})$, one has
	\[
	f_k(x_{n'},\Omega_{n'})=v_k(\Omega_{n'}) \quad \text{for all } n'.
	\]
	By the continuity of $f_k$ and $v_k$, it follows that
	\[
	f_k(\overline{x},\Hat\Omega)
	=\lim_{n'\to\infty} f_k(x_{n'},\Omega_{n'})
	=\lim_{n'\to\infty} v_k(\Omega_{n'})
	=v_k(\Hat\Omega).
	\]
	Thus,
	\[
	\overline{x}\in S_k(\Hat\Omega)\subset V.
	\]
	On the other hand, since $x_{n'}\notin V$ for all $n'$ and $\mathbb{R}^{dk}\setminus V$ is closed, we must have
	\[
	\overline{x}\in \mathbb{R}^{dk}\setminus V,
	\]
	which is a contradiction. Therefore, $S_k$ is upper semicontinuous at $\Hat\Omega$. 
\end{proof}

\color{black}
\begin{proposition}
	\label{G.S.Mapping_5}
	Let $\Hat\Omega\in \mathcal{K}^m $ be arbitrary. If $S_k(\Hat \Omega)$ is compact, then there exists a neighborhood $U$ of $\Hat \Omega$ such that $S_k(\Omega)$ is compact for every $\Omega \in U$.
\end{proposition}
\begin{proof}
	Define $v:\mathcal{K}^m \to \R$ by
	\begin{equation}
		\label{G.S.Mapping_5_1}
		v(\Theta)=v_{k-1}(\Theta)-v_k(\Theta) \text{ with } \Theta \in \mathcal{K}^m.
	\end{equation}
	Since $S_k(\Hat \Omega)$ is compact, it follows from Theorem \ref{G.O.Solutions_5} and \eqref{v_k_definition} that
	\begin{equation}
		\label{G.S.Mapping_5_2}
		v(\Hat \Omega)=\min_{y^{-k} \in \mathbb{R}^{d(k-1)}}f_k(y^{-k},\Hat\Omega)-\min_{y \in \mathbb{R}^{dk}}f_k(y,\Hat\Omega)>0.
	\end{equation}
	Fix
	\[
	\varepsilon=\frac{v(\Hat\Omega)}{2}>0.
	\]By Proposition \ref{O.V.F_Lipschitz_4}, the function $v(\cdot)$ is continuous on $(\mathcal{K}^m,d_{\mathcal{K}^m})$. Hence, there exists $\delta(\varepsilon)>0$ such that for every $\Omega \in \mathcal{K}^m$ satisfying
	\[
    d_{\mathcal{K}^m}(\Omega;\Hat\Omega)<\delta(\varepsilon),
    \]
    we have
	\begin{equation}
		\label{G.S.Mapping_5_3}
		\bigl|v(\Omega)-v(\Hat\Omega)\bigr|<\varepsilon. 
	\end{equation}
	Let 
	\[
	U=\mathbb{B}_{d_{\mathcal{K}^m}}\bigl(\Hat\Omega;\delta(\varepsilon)\bigr)=\bigr\{ \Theta\in \mathcal{K}^m\mid d_{\mathcal{K}^m}(\Theta,\Hat\Omega)<\delta(\varepsilon)\bigl\}.
	\]
	Combining \eqref{G.S.Mapping_5_2} and \eqref{G.S.Mapping_5_3}, one obtains
	\[
	v({\Omega})>v(\Hat\Omega)-\varepsilon>0,  \text{ for all } \Omega \in U.
	\]
    From definition \eqref{G.S.Mapping_5_1}, we have
	\[
	v_{k-1}(\Omega)>v_k(\Omega), \quad \text{for all } \Omega \in U.
	\]
	According to Theorem \ref{G.O.Solutions_5}, $S_k(\Omega)$ is compact for all $\Omega \in U$. Therefore, $U$ is the neighborhood of $\Hat\Omega$ that we are looking for, and the proof is complete.
\end{proof}

\subsection{Local Solution Mappings}

\begin{definition}
	\label{L.S.Mapping_1}
	The set-valued mapping $S_k^{\mathrm{loc}}:(\mathcal{K}^m,d_{\mathcal{K}^m})
	\rightrightarrows \mathbb{R}^{dk}$, defined by
	\begin{equation}
		\label{S_kloc_definition}
		S_k^{\mathrm{loc}}(\Omega)=
		\left\{
		x\in\mathbb{R}^{dk}\ \middle|\ 
		\exists\,\varepsilon>0 \text{ such that }
		f_k(x,\Omega)\leq f_k(y,\Omega)
		\text{ for all } y\in \mathbb{B}(x,\varepsilon)
		\right\},
	\end{equation}
	is called the \textit{local solution mapping} of problem~\eqref{Generalized_Multi_Source_Weber_Problem}. 
\end{definition}

\begin{proposition}
	\label{L.S.Mapping_2}
	Let $\Hat x\in S_{k}^\mathrm{loc}(\Hat{\Omega})$ be a strict local optimal solution of problem~\eqref{Generalized_Multi_Source_Weber_Problem} at $\Hat{\Omega}\in\mathcal{K}^m$. Then the local solution mapping $S_{k}^\mathrm{loc}$ is inner semicontinuous at $(\Hat{\Omega},\Hat{x})$.
\end{proposition}

\begin{proof}
	Suppose on the contrary that $S_k^{{\rm loc}}$ is not inner semicontinuous at $(\Hat\Omega;\Hat{x})\in \gph(S_k^{{\rm loc}})$. Then, by definition, there exists an open neighborhood $V$ of $\Hat{x}$ such that for every neighborhood $U$ of $\Hat\Omega$, there exists some $\Omega\in U$ satisfying
	\[
	S_k^{{\rm loc}}(\Omega)\cap V=\emptyset.
	\]
	For each $n\in\mathbb N$, let
	\[
	U_n=\mathbb{B}_{d_{\mathcal{K}^m}}\left[\Hat\Omega;\frac{1}{n}\right]:=\biggr\{ \Omega\in \mathcal{K}^m\mid d_{\mathcal{K}^m}(\Omega,\Hat\Omega)\leq\frac{1}{n}\biggl\}.
	\]
	Then each $U_n$ is a neighborhood of $\Hat\Omega$, and we can find a sequence $\{\Omega_n\}\subset \mathcal{K}^m$ such that $\Omega_n\in U_n$, $\Omega_n\to \Hat\Omega$, and
	\begin{equation}
		\label{L.S.Mapping_2_1}
		S_k^{{\rm loc}}(\Omega_n)\cap V=\emptyset
		\quad \text{for all } n\in\mathbb N.
	\end{equation}
	
	Since $\Hat x$ is a strict local optimal solution at $\Hat\Omega$, there exists $\varepsilon>0$ such that $\Hat x$ is the unique optimal solution of $f_k(\cdot,\Hat\Omega)$ on $\mathbb B\left(\Hat x;\varepsilon\right)$. Without loss of generality, we may choose $\varepsilon$ small enough so that
	\[
	\mathbb B\left[\Hat x;\frac{\varepsilon}{2}\right]\subset V.
	\]
	For each $n\in\mathbb N$, consider the restricted optimization problem
	\[
	\min\left\{f_k(x,\Omega_n)\mid x\in\mathbb B\left[\Hat x;\frac{\varepsilon}{2}\right]\right\}.
	\]
	Since $f_k(\cdot,\Omega_n)$ is continuous and $\mathbb B\left[\Hat x;\frac{\varepsilon}{2}\right]$ is compact, the Weierstrass extremum theorem ensures the existence of a point
	\[
	x_n\in \mathbb B\left[\Hat x;\frac{\varepsilon}{2}\right]
	\]
	such that
	\begin{equation}
		\label{L.S.Mapping_2_2}
		\min_{x\in\mathbb B\left[\Hat x;\frac{\varepsilon}{2}\right]}f_k(x,\Omega_n)
		=
		f_k(x_n,\Omega_n).
	\end{equation}
	
	By the compactness of $\mathbb B\left[\Hat x;\frac{\varepsilon}{2}\right]$, there exists a subsequence $\{x_{n'}\}$ converging to some $\overline{x}\in \mathbb B\left[\Hat x;\frac{\varepsilon}{2}\right]$. Since $x_{n'}$ is an optimal solution of $f_k(\cdot,\Omega_{n'})$ on $\mathbb B\left[\Hat x;\frac{\varepsilon}{2}\right]$, we have
	\[
	f_k(x_{n'},\Omega_{n'})
	\leq f_k(x,\Omega_{n'})
	\quad \text{for all } x\in\mathbb B\left[\Hat x;\frac{\varepsilon}{2}\right].
	\]
	Because $f_k$ is continuous on $\mathcal X$, letting $n'\to\infty$ yields
	\[
	f_k(\overline{x},\Hat\Omega)\leq f_k(x,\Hat\Omega)
	\quad \text{for all } x\in\mathbb B\left[\Hat x;\frac{\varepsilon}{2}\right].
	\]
	This means that $\overline{x}$ is an optimal solution of $f_k(\cdot,\Hat\Omega)$ on $\mathbb B\left[\Hat x;\frac{\varepsilon}{2}\right]$. Since $\Hat x$ is the unique optimal solution of $f_k(\cdot,\Hat\Omega)$ on $\mathbb B\left[\Hat x;\varepsilon\right]$, we must have
	\[
	\overline{x}=\Hat x.
	\]
	Hence $x_{n'}\to \Hat x$, and so
	\[
	x_{n'}\in \mathbb B\left[\Hat x;\frac{\varepsilon}{2}\right]\subset V
	\quad \text{for all } n' \text{ sufficiently large.}
	\]
	Together with \eqref{L.S.Mapping_2_2}, this implies that $x_{n'}$ is a local optimal solution of $f_k(\cdot,\Omega_{n'})$ for all $n'$ sufficiently large, that is,
	\[
	x_{n'}\in S_k^{{\rm loc}}(\Omega_{n'}).
	\]
	Therefore,
	\[
	S_k^{{\rm loc}}(\Omega_{n'})\cap V\neq\emptyset
	\quad \text{for all }n' \text{ sufficiently large,}\]
	 which contradicts \eqref{L.S.Mapping_2_1}. The proof is complete.
\end{proof}

\section{Conclusion}\label{conclu}

We have studied the generalized multi-source Weber problem with set-valued targets via minimal time functions. Our analysis established the existence of optimal solutions, several qualitative properties of the solution sets, Lipschitz continuity of the objective and optimal value functions, and stability properties of the global and local solution mappings. These results contribute to the qualitative theory of generalized location problems and may serve as a basis for future numerical and stability analysis.


\end{document}